\documentclass[12pt]{elsarticle}
\usepackage{makeidx}     
\usepackage{enumerate} % allows index generation
\usepackage{graphicx,epsfig,lscape,float}    % standard LaTeX graphics tool
\usepackage{amsmath,array,amssymb, amsfonts, latexsym, mathrsfs, dsfont,amsthm}

\usepackage{psfrag,fancybox}%,authblk} <=== steht im konflickt mit elsarticle !!!!!!!!!! 

\usepackage{tikz}
\usetikzlibrary{shapes,decorations}
\usetikzlibrary{arrows}
\usepackage{rotating}
\usepackage{pgfplots}
\usepgfplotslibrary{units}

\usetikzlibrary{calc}

\usepackage{multicol}    % used for the two-column index
\makeindex             % used for the subject index
\newcommand{\cC}{{\mathcal C}} 
\newcommand{\cD}{{\mathcal D}}

\newcommand{\sign}{\mathbf{sign}}

\newcommand{\be}{\begin{equation}}
\newcommand{\ee}{\end{equation}}

\newcommand{\R}{\mathbb{R}}
\newcommand{\N}{\mathbb{N}}
\newcommand{\tx}{\mathring{x}}

\newcommand{\restr}[1]{\lower0.4ex\hbox{$\vert$}\lower0.7ex\hbox{ $\!_{#1}$ }}

\DeclareMathOperator{\diag}{\mathbin{\operatorname{diag}}}

\newtheorem{proposition}{Proposition}[section]

\newtheorem{corollary}[proposition]{Corollary}
\newtheorem{lemma}[proposition]{Lemma}
\newtheorem{definition}[proposition]{Definition}

\usepackage{amsmath,array,amssymb,amsfonts,graphicx}
\journal{Linear Algebra and its Applications}

\begin{document}

\title{  Solving piecewise linear equations        %\tnoteref{t1}
 $\,$\\ in abs-normal form  
}

 \author{ Andreas Griewank, Jens-Uwe Bernt,  Manuel Radons  and Tom Streubel }
\address{ Department of Mathematics, Humboldt-Universit\"at zu  Berlin\\
 e-mail: surname@math.hu-berlin.de
\vspace*{-15pt}} 
\date{ Draft Version of \today}

\maketitle

\textbf{Abstract:}
{\em
With the ultimate goal of iteratively solving  piecewise smooth (PS) systems, we
consider the solution of  piecewise linear (PL) equations.
As shown in \cite{griewank2013stable} PL models can be derived in the fashion of
automatic or algorithmic differentiation as local approximations of PS functions with a
second order error in the distance to a given reference point. The
resulting PL functions are obtained  quite naturally in what we call the
abs-normal form, a variant of the state representation proposed by
Bokhoven in his dissertation \cite{van1981piecewise}.  Apart from the tradition
of PL modelling by electrical engineers, which dates back to the
Master thesis of Thomas Stern \cite{MR0086600} in 1956, we take into
account more recent results on linear complementarity problems and semi-smooth equations 
originating in the optimization community \cite{0757.90078,
scholtes2012introduction, pang1992linear}.  We analyze simultaneously the original PL problem (OPL)
in abs-normal form and a corresponding complementary system (CPL), which is closely related to 
the absolute value equation (AVE) studied by Mangasarian et al \cite{mama} and 
a corresponding linear complementarity problem (LCP). We show that the CPL, like KKT conditions and other
simply switched systems, cannot be open without being injective. Hence some of the intriguing PL structure described by Scholtes 
in
\cite{scholtes2012introduction} is lost in the transformation from OPL to CPL.  To both
problems one may apply  Newton variants with appropriate generalized
Jacobians directly computable from the abs-normal representation.
Alternatively, the CPL can be solved by Bokhoven's modulus method and
related  fixed point iterations.  We compile the properties of the
various schemes and highlight the connection to the properties of the
Schur complement matrix, in particular its signed real spectral radius as
analyzed by Rump in \cite{rump1997theorems}. Numerical experiments and suitable
combinations of the fixed point solvers and  stabilized generalized Newton
variants remain to be realized.   \em} 
\section{Introduction and Motivation}
In many applications one encounters piecewise smooth (PS) functions that can be approximated locally with second order error
by piecewise linear (PL) functions. 
In this paper we will assume throughout that all functions are continuous and thus, in fact, Lipschitz continuous. However, an extension  to piecewise linear but possibly discontinuous problems should be in the back of our  minds before we settle on data structures and interfaces. Discontinuous solution operators   may arise for example, if one considers least squares problems defined by piecewise linear systems of equations.       

The process of piecewise linearization of a piecewise smooth function $F: \cD \subset \R^n\mapsto \R^m$ given by an evaluation procedure was described in 
\cite{griewank2013stable}. The key assumption is that all nonsmoothness can be cast in terms of the absolute value function $|\cdot | $. Then piecewise linearization can be achieved in the style of algorithmic differentiation \cite{griewank2008evaluating} by simply replacing all smooth elemental functions by their tangent line or plane (in case of binary operations or special functions) and the absolute value function by itself. 

In contrast to conventional notions of differentiation one does not obtain a collection of derivative vectors or matrices at a given reference point $\tx$. 
Rather one arrives at a procedure for evaluating an incremental PL function 
$\Delta  F(\tx, \Delta x) : \cD \times \R^n\mapsto \R^m$  for which  
$$ F(\tx + \Delta x ) \; = \; F(\tx) + \Delta  F(\tx, \Delta x)  + O(\|\Delta x\|^2)  \; .$$  
Here the error term $\|\Delta x\|^2$ is uniform on compact subsets of $\cD - \tx$.  
This means that $\Delta  F(\tx, \Delta x)$ is a candidate for a nonsingular uniform Newton approximation in the sense of \cite{pang1992linear}, although the local homeomorphism property is by no means guaranteed.

Throughout this paper we will only be concerned  with the properties of the piecewise linearized function. 
% and hence assume without loss of generality that $\tx =0$. 
We will also drop the decomposition into $F(\tx)$ and the increment $\Delta F(\tx,\Delta x)$ and thus simply consider a globally defined piecewise linear continuous (PL) mapping 
$$ F(x) \; : \;  \R^n \mapsto \R^m  \; .$$   
Like  for the (possibly) underlying nonsmooth  mapping, our ultimate purpose is to solve certain basic numerical tasks, in particular (un)constrained optimization, equation solving, and the numerical integration of dynamical systems.  Here we will consider, for $m=n$, the problem of solving the formally well determined system of equations  
\begin{equation}
 F(x)  \; = \; 0 \in \R^n , \quad \text{for} \quad x \in \R^n. 
\label{probs} 
 \end{equation}
 The paper is organized as follows: In Section 2 we introduce PL functions $F$ in {\em abs-normal} form, a term that was apparently introduced  by Barton and Khan in a more general nonlinear setting \cite{springerlink}.   In Section 3 we describe the resulting polyhedral structure and give an explicit procedure for calculating generalized  Jacobians of $F$, which were shown in  \cite{springerlink, Khan:2013:EEC} and \cite{griewank2013stable} to be conically active limiting Jacobians of the underlying piecewise smooth function, whenever $F$ was obtained as its piecewise linearization. In Section 4  we examine the relation between the global properties of bijectivity and coherent orientation, which coincide under certain rather generic conditions.  Section 5 discusses sufficient conditions for the global convergence of the generalized Newton method, which is often referred to as semi-smooth Newton. In Section 6 we unfold the system by elevating the intermediate switching variables to the status of full 
variables. As is the case for the unfolding of smooth singular equations \cite{golubitzky}, in this process some regularity is gained, but some information is also lost. The resulting system, that we call the complementary piecewise linear system (CPL), is always  simply switched and as shown in Section 7, it can be solved by two different 
 fixed point methods and several variants  of generalized Newton.    Finally, the complementary system can also be rewritten as a linear complementarity problem \cite{0757.90078} with coherent orientation being equivalent to the P-matrix property.   The final \mbox{Section  8} summarizes our results and provides an outlook to further developments.         
 
\section{The abs-normal form}
As also observed by Scholtes in \cite{scholtes2012introduction} any piecewise linear scalar function $f : \R^n \mapsto \R$ 
has a so-called \textbf{max-min} representation  
$$ f(x) \;  = \; \max_{1 \leq i \leq l}  \,  \min_{j \in M_i} a_j^\top x + b_j $$
where the $l$ index sets $M_i$ are contained  in $\{1,2 \ldots k\}$ for some $k \in \N$ and the $a_i \in \R^n$, $b_i \in \R$ are constant coefficients.
For a PL vector function $F : \R^n \mapsto \R^m$ each one of the $m$ component functions can be represented in the same way. Moreover,  
using the equivalences  
$$ \max(u,w) = \tfrac{1}{2}(u+w+|u-w|) \quad \mbox{and} \quad  \min(u,w) = \tfrac{1}{2}(u+w-|u-w|) $$   
one can express all \textbf{min} and \textbf{max} expressions in terms of $s \geq 0$ absolute value functions $| z_i |$, whose arguments $z_i$ are called {\em switching variables.} 

Observing that each $z_i$ is an  affine function of absolute values $|z_j|$ with $j < i$ and the independents $x_k$ for $k \leq n $, one arrives at an \textbf{abs-normal} representation
  \begin{eqnarray}
 \begin{bmatrix} z  \\ y  \end{bmatrix} \;  = \; \begin{bmatrix} c \\  b \end{bmatrix} +   \begin{bmatrix} Z & L \\ J & Y   \end{bmatrix} \; 
\begin{bmatrix} x \\  |z|  \end{bmatrix} \label{absnormal}  \; .
\end{eqnarray} 
Here the two vectors and four  matrices specifying the function $F$ have the formats
$$ c \in \R^s, \;  Z \in \R^{s\times n}, \;  L \in   \R^{s\times s}, \;   b \in \R^m, \; J \in   \R^{m\times n}, \;   Y \in   \R^{m\times s}. $$ 
The matrix $L$ is strictly lower triangular so that for given $x$ the components of $z=z(x)$ and thus $|z|$ can be unambiguously  computed one by one. Specifically, we have $L_{i, j} \, \neq \, 0$ exactly if $z_i$ depends directly on $|z_j|$ so that there is an edge between the  nodes $j$ and $i$ in the corresponding data dependency graph. This graph is always acyclic and the components of $x$, $y$ and $z$ represent its roots, leaves and internal vertices, respectively.  

Of course, the  representation \eqref{absnormal} is by no means unique for a given mapping $F$. One would naturally strive to make the representation as concise as possible in some sense.  
Excluding incidental cancellations, we find that the smallest integer $\nu \leq s $ for which $$ L^\nu \; = \; 0  $$
corresponds to the maximal number of internal nodes in any chain in the data dependency graph.  We will call this the  
{\bf switching depth} and consider it as key measure of the combinatorial difficulty of the function $F$. In this terminology, $F$ is fully
linear  exactly if $\nu =0$ with $s=0$ and thus $z$, $Z$, and $L$ are empty. We will refer to  this limiting situation as the {\bf smooth case.} 
We will call $F$ {\bf simply switched} if $\nu=1$,  a situation that arises for example in complementarity problems, where none of the nonsmooth elements are superimposed. We conjecture that, for any PL mapping $F: \R^n \mapsto \R^m $, there is an abs-normal representation with a switching depth 
$\nu \leq \bar \nu(n) = 2\, n-1$. 

Formulations similar to our abs-normal form have been used for a long time
in the engineering literature. In \cite{van1999explicit,leenaerts} several classes of
PL models are compared, \textbf{Chua1} has switching depth 1 and
\textbf{Gr\"u} as well as \textbf{Bokh2} are limited to switching depth 2.
It is shown there that all of them are specializations of the model
\textbf{Bokh1}, which is a priori implicit in that evaluating $y$ for
given $x$ requires the solution of an LCP with a system matrix $D$.
However, if $D$ is also lower triangular solving the LCP requires simply a
forward substitution. Then, provided $D$ is nonsingular, the intermediate variables $z$ can be rescaled such that $D-I$ and consequently the 
M\"obius transform $L = (I+D)^{-1}(I-D)$ of $D$
become strictly lower triangular. $L$ then defines an abs-normal form equivalent to the \textbf{Bokh1} system.

Mangasarian and Meyer also observed in \cite{mama} the connection between LCPs and what 
they call an absolute value equation (AVE), the concept of which is closely related to our complementary system (CPL). 
We will partly replicate and strengthen their result.  
As we have noticed, the abs-normal form is general enough
to represent all continuous PL functions, so we will not use the even
greater generality of the implicit \textbf{Bokh1} model.  

In  the more mathematical literature,  piecewise linear systems 
are often specified by linear pieces on simplices defined by systems of linear inequalities. These approaches may also be interpreted as 
conjunctive programming or mixed integer nonlinear programs (MINLP) as in  \cite{martin}. However, these representations tend to 
be of combinatorial complexity and highly redundant, whereas 
the abs-normal form is stable and completely free of redundancy. In particular, any perturbation of the  four matrices 
 $Z, L, J$ and $Y$ that preserves the strict lower triangularity of $L$ again unambiguously defines a  
 continuous   PL function $y=F(x)$.

In the simply switched case we have $z = c+ Z x$, which means that potential kinks occur at the union of the $s$ hyperplanes 
$z_i(x) = c_i + e_i^\top Z x=0 $ for $i=1\ldots s$. We will then say that the kinks satisfy the \textbf{linear independence kink qualification LIKQ}  if the normals of the hyperplanes intersecting at some point $x$ are always linearly independent. This implies in particular that the vector $z = c+Zx$ can never have more than $n$ vanishing components. LIKQ is implied by  all square  submatrices of $[c,Z] \in \R^{s \times (1+n)}$ of order $\min(s,n+1)$ being nonsingular.  That slightly stronger condition is for example satisfied if
$c=\mathbf{1}$ is  the vector of ones and $Z=(\lambda_i^j)^{i=1\ldots
s}_{j=1\ldots n}$ is a Vandermonde matrix at distinct abscissas
$\lambda_i$ for $i=1\ldots s$. Consequently,  the polynomial $P(c,Z)$ formed by the
product of the determinants of all maximal square submatrices $[c,Z]$ does
not vanish at the Vandermonde choice and the same is true for 
almost all matrices $[c,Z] \in \R^{s \times (1+n)}$. In other words, LIKQ is a generic property, like linear independence of active constraints in linear optimization (LOP). 

\subsection*{The Rosette example}
%ƒ\new page
To highlight the possible properties of PL functions we take a look at the following class of examples. Positively homogenous functions in two variables are uniquely defined by their values on the unit circle, which must be $2\pi$ periodic functions of the  polar angle $\varphi(x)  = \arctan(x_1,x_2)$.  More specifically, we assume that we have a monotonically growing sequence of angles 
$$  0 = \varphi_0 < \varphi_1 < \ldots   < \varphi_{n-1} <   \varphi_{n} = 2  \pi $$
and corresponding values 
$$(\psi_i)_{i= 0\ldots n}  \; \mbox{with}  \;  \psi_n - \psi_0 =  2  p  \pi \;  \mbox{for} \;  p\in \N  \; .$$
By suitable subdivisions we can ensure that the increments  $\varphi_i - \varphi_{i-1}$ and $|\psi_i - \psi_{i-1}|$ are all less than $\pi$.  
Then there exists a homogenous piecewise linear function $F : \R^2 \mapsto \R^2$ such that
$$ F(\cos \varphi_i,  \sin{\varphi_i})   \; =  \;  (\cos \psi_i, \sin{\psi_i})  \quad \mbox{for} \quad i = 0 \ldots n \; .$$
We can make \(F\) unique by minimizing the number of linear pieces through the natural requirement that $F$ is linear on the sectors \[ \big\{ (x,y) \in \R^2 \mid \varphi_i < \arctan\left(\tfrac yx \right) < \varphi_{i+1} \big\} .\]
As shown in Figure \ref{fig:triangles} the function $F$ can be visualized as a mapping between the triangles 
$(0,0), (\cos \varphi_{i-1} , \sin{\varphi_{i-1}}), (\cos \varphi_{i}, \sin{\varphi_{i}})$ in the domain
and the triangles \\ $(0,0), (\cos \psi_{i-1}, \sin{\psi_{i-1}}), (\cos \psi_{i}, \sin{\psi_{i}})$ in the range. By imposing certain conditions on the angles 
$\psi_i$ we can ensure certain properties of the resulting $F$. More specifically, 
% if \(\forall \theta \in{[0,2\pi[}:\, \psi(\theta)\ge\theta\) (i.e. every vector will be turned clockwise)
the following implications hold true
 \begin{center}
\begin{tabular}{lcl}
	$\psi_i$ strictly monotone and $p=1$  &\(\implies\)& \(F\) injective, \\
	$\psi_i$ strictly monotone and $p>1 $  &\(\implies\)& \(F\) not injective but open, \\
	$\psi_i$ are not monotone  but \(p > 0\) &\(\implies\)& \(F\) not  open but surjective.
\end{tabular}
\end{center}
\begin{figure}[H]
\includegraphics[width= 0.9\linewidth]{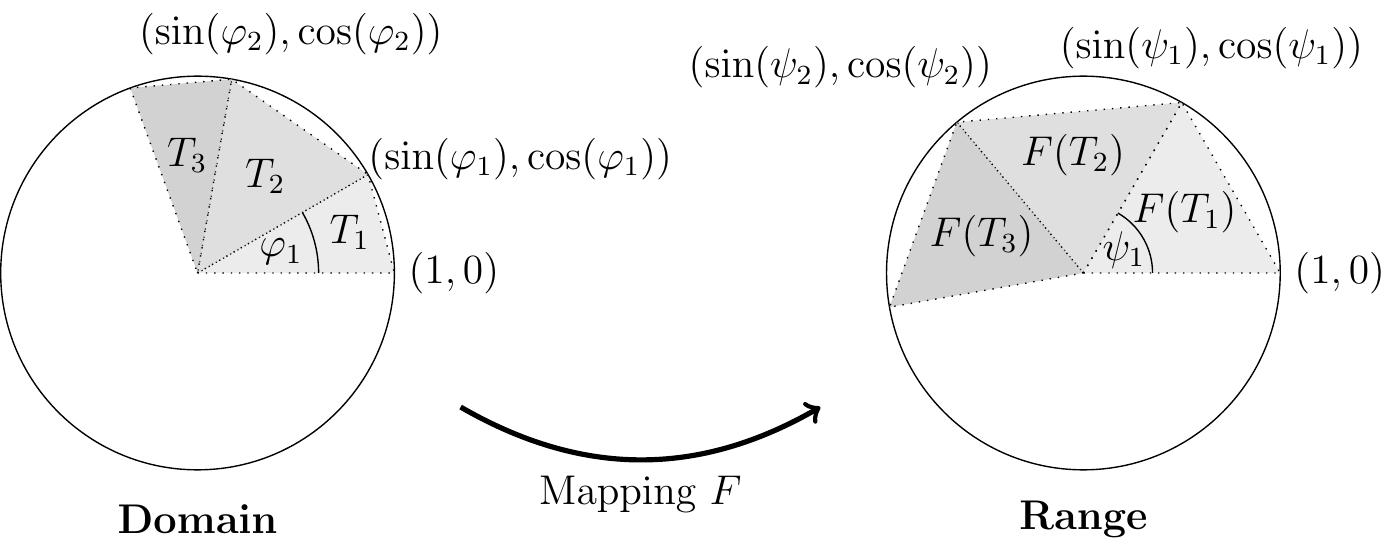}
\caption{Piecewise linear Rosette Example on $\R^2$} 
\label{fig:triangles}
\end{figure}
 
In other words, we have a simple class of examples, which demonstrate that the well known chain of implications \cite{scholtes2012introduction}
\be \mathbf{bijective} \quad \Longleftrightarrow\quad  
 \mathbf{injective} \quad \Longrightarrow\quad  
 \mathbf{open} \quad \Longrightarrow\quad   \mathbf{surjective} \label{chain0} \ee
for general PL  functions cannot be strengthened. Here openness means that all images $y=F(x)$ are in the interior of $F(B_r(x))$ 
for any ball $B_r(x)$ about any preimage $x$ of $y$. Moreover, in the PL case openness is equivalent to coherent orientation, i.e., the property that the determinants of all linear pieces have the same nonzero determinant sign.  In the context of the abs-normal form we can verify this important property more or less explicitly as follows.

\section{Polyhedral structure and limiting Jacobians}
As in \cite{griewank2013stable} we define the signature vector and matrix by  
$$ \sigma \equiv  \sigma(x)  \equiv  \mathbf{sign}(z(x)) \in \{-1,0,1\}^s  \quad  \mbox{and} \quad     \Sigma  \equiv  \Sigma(x) \equiv  \mathbf{diag}(\sigma) \; \in \{-1,0,1\}^{s \times s}. $$
This vector maps $\R^n$ into $\{-1,0,1\}^s$ and represents the control flow in our calculation. As an aside we note that all possible sign combinations must indeed occur if $Z$ is surjective, which requires $s \leq n$ so that there may actually occur $3^n$ different signatures. As in \cite{griewank2013stable} one can verify that the corresponding sets 
$$ P_\sigma \; \equiv \; \{ x \in \R^n : \sigma(x) = \sigma \}     $$
are relatively open and convex polyhedra in $\R^n$. Being  inverse images they are mutually disjoint and span the whole domain $\R^n$. By continuity it follows that  $P_\sigma$ must be open
(possibly empty) if $\sigma$ is definite in that all its components are nonzero.  In degenerate situations there may be some indefinite $\sigma$ that are nevertheless {\bf  open} in that 
 $P_\sigma$ is open. 
 
The limiting Jacobian $\partial^L\!  F(x)$ at some $x \in \R^n$, i.e., the limits of all proper Fr\'echet derivatives in its  neighborhood, is in the PL case simply the finite set 
$$   \partial^L\! F(x) \; = \; \{ J_\sigma : x \in \overline P_\sigma \, \mbox{with} \; \sigma \; \mbox{open}  \} \; . $$ 
The Clarke generalized Jacobian is the convex hull $\partial F(x) = \mathbf{conv} \left ( \partial^L \! F(x) \right ) $. In general it will be quite difficult to calculate all elements of the generating set $  \partial^L \! F(x)$ and we will usually shy away from that combinatorial effort.  

\subsection*{Explicit Jacobian representation}

 On all {open} $\sigma$ we find that $|z| = \Sigma z $, so that the first equation in \eqref{absnormal} yields 
 $$ (I- L\Sigma) z = c + Z x   \quad \mbox{and} \quad z \; = \; (I- L \Sigma)^{-1}(c+Z x)   \; . $$
 Notice that due to the strict triangularity of $L\,\Sigma$ the inverse of $(I- L \Sigma)$ is well defined and polynomial in the entries of $L$. 
  Moreover, due to the structural nilpotency degree $\nu$  of $L$ we obtain the Neumann expansion 
  \begin{eqnarray}
  (I- L \Sigma)^{-1}  \; = \; I + L \Sigma +  (L \Sigma)^2 + \cdots  + (L \Sigma)^{(\nu -1)} \; . 
  \end{eqnarray}
 In the simply switched case $\nu = 1$ we have $L=0$ and thus the expansion reduces to $I^{-1}=I$. When $\nu=2$, we have the linear inverse 
 $  (I- L \Sigma)^{-1} = I + L \Sigma $.  Substituting this expression into the second part of \eqref{absnormal} we obtain the local representation:
 \begin{proposition}    
 On all open  $P_\sigma$ the dependents $y$ can be directly expressed in terms of $x$, namely as   
 \begin{equation}
  y \; = \; b + Y \Sigma (I- L \Sigma)^{-1} c  + J_\sigma \,  x \quad  \mbox{with} \quad    J_\sigma  =  J  + Y \Sigma(I-L\Sigma)^{-1}Z .  \label{genjac}
  \end{equation}  
 Here  $J_\sigma$ is the Jacobian of $F$ restricted to $P_\sigma$. It reduces to $J_\sigma  = J + Y \Sigma Z$ for simply switched problems ($\nu=1$) and to $J$ for 
 smooth problems ($\nu=0$). 
 \label{genjacprop}
\end{proposition} 
\subsection*{Polynomial escape}
Computing generalized Jacobians  $J_\sigma$ according to \eqref{genjac} is
quite simple, once an open signature $\sigma$ and thus the corresponding
diagonal $\Sigma$ are known.
To find, for a given $x$, some open $\sigma$  with the closure  $\bar
P_{\sigma}$ containing $x$ one may use the following trick, which we like to call
{\bf polynomial escape}.  Due to piecewise linearity the complement $\cC$
of all open $P_\sigma$
is contained in the union of finitely many hypersurfaces.
Hence, no polynomial path of the form
$$     x(t) \; \equiv \; x + \sum_{i=1}^n \hat e_i t^i  \quad \mbox{with} \quad
\det \left [\hat e_1, \hat e_2, \ldots, \hat e_n \right ] \neq 0 , \quad \mbox{for} \quad
\hat e_i \in \R^n$$
can be contained  in $\cC$. In other words, we find for some $\sigma$ and
$\bar t > 0$ that $x(t) \in P_\sigma$ for all $t \in (0,\bar t)$.
The corresponding $\sigma$ can be computed by lexicographic
differentiation as introduced by Nesterov \cite{nesterov2005lexicographic} and described in
a little more detail in \cite{griewank2013stable}. There it is also shown that
any such $J_\sigma$ is in fact a generalized Jacobian of the underlying
nonlinear function if $F$ was obtained by piecewise linearization.  
Finally, by suitably selecting $\hat e_1=d \neq 0 $, one can make sure that the
generalized Jacobian obtained is active in  a cone containing the given
direction $d$ at least in its closure.

\section{Coherent Orientation and Injectivity}
As in the smooth case, the determinants of the Jacobains $J_\sigma$ are of crucial importance for the properties of the PL function $F:\R^n\to\R^n$. 
It is called coherently oriented if all its Jacobians have the same nonzero determinant sign. As stated for example in \cite{scholtes2012introduction}, the central property 
openness in the  chain \eqref{chain0} is, for PL functions, equivalent to coherent orientation.   
For simply switched $F$, like for example all KKT systems of  QOPs, we have essentially the same situation as in the affine case, namely bijectivity follows already from  coherent orientation and LIKQ. 
\begin{proposition}\label{SS+LIKQ}
If $F$ is simply switched in that $L=0$ and its kinks satisfy LIKQ then $F$ is bijective if and only if it is  coherently oriented.
\end{proposition}
\begin{proof}
If \(F\) is bijective it follows from Scholtes' chain of implications \eqref{chain0} that it is already coherently oriented. For the inverse direction: On the basis of the mean value theorem, see Prop$7.1.16$ in \cite{pang1992linear}, Clarke showed that  $F$ has an inverse function near  some
point $x$ if all elements of the generalized
Jacobian $\partial F(x)$ are nonsingular. At all points where $F$ is
differentiable this follows from the assumed coherent orientation.
At all other points a certain number of $m \leq s$ components of $\sigma$
vanish, which means in the simply switched case that the
$s-$vector $c+Z x$ has $m$ zero components. In fact, it may contain at
most $m \leq n$ zeros since otherwise a corresponding
$(n+1) \times (n+1)$ sub-matrix of $[c,Z]$ would have the nonzero null
vector  $(1,x^\top)^\top \in \R^{n+1}$. Without loss of generality we may
assume that exactly the first $m \leq n$ components of $z=z(x)$ vanish.
The remaining ones will keep their sign in a sufficiently small
neighborhood of $x$. Due to the linear independence of the first $m$ rows of $Z$ we can
find  arbitrarily small perturbations $\Delta x \in \R^n$ such that
the first $m$ components of $c+Z(x+\Delta x)$ have any one of $2^m$ sign
patterns. Correspondingly, the first $m$ components of the signature vector
$\sigma \in \R^s$ attain any \(\{-1,1\}\) pattern  on some open domain whose
closure contains the given points $x$. Hence, $\partial F(x)$ contains all
matrices
$J_\sigma = J +Y \Sigma Z$ where the last $s-m$ components of $\Sigma$ are
fixed and the first $m$ may be $+1$ or $-1$. By assumption, all these
$J_\sigma$ have the same determinant sign. Changing just one $\sigma_i\in\{-1,0,1\}$ of
the first $m$ components continuously from $-1$ to $+1$ corresponds to
a rank one change in the corresponding matrix $J_\sigma$, whose
determinant varies linearly with respect to $\sigma_i$ and therefore
cannot change
signs in between.  Thus the $J_\sigma$ along all edges have the same
determinant signs, which are inherited by the ones on the face and so on.
Therefore, we have shown that all generalized Jacobians  are nonsingular so that
$F$ is everywhere locally injective and  also globally injective.
 \end{proof} 

\begin{lemma}
  Any $F$ satisfying the assumptions of the proposition is {\em stably coherently oriented} in that all modifications generated by  small perturbation of $[c,Z]$ are also coherently oriented. 
\end{lemma}

\begin{proof}
We firstly note that each open polyhedron of the original system is a simplex whose vertices are intersections of exactly 
$n+1$ linearly independent hypersurfaces. Hence for sufficiently small perturbations of the data each of them persist and remain nondegenerate. 
Moreover, the determinant of the also continuously varying Jacobians maintain the same sign. Now suppose  some arbitrarily small perturbations had an additional 
open polyhedron, for which we may assume without loss of generality the same definite signature $\sigma$, due to the finiteness of the whole situation. Then the 
corresponding polyhedron $P_\sigma$ of the original problem must be nonempty but nonopen. That means the linear inequalities active at any one of its elements must be 
linearly dependent in violation of LIKQ. \end{proof}
      The converse is not true, since one may modify any $F$ with an unstable decomposition at $x$ into one that is stably 
coherently oriented by adding a suitable multiple $\alpha$ of the identity so that $F(x)$ becomes $F(x)+\alpha x$. This modification does not affect $z$ and 
thus the lack of LIKQ. 

As we have seen the Rosette example may be open but not injective, which is not surprising since it has the switching depth 2 and is not stably coherently oriented.     
Just assuming stable coherent orientation, we find that 
all the small perturbations satisfying LIKQ are injective and $F$, as the limit of such bijective
perturbations, inherits this property by the proposition that follows from the lemma below. 
\begin{lemma}\label{lemm42}
Let $\cD\subseteq\R^n$ be open, and let
$\{F_k\}$ be a sequence of continuous injective maps $F_k:\cD\to \R^n$
which converges uniformly on compact sets to~$F:\cD\to\R^n$. 
Then for every $x_0\in \cD$ and every $\varepsilon>0$ with
$B_\varepsilon(x_0)\subseteq \cD$
there exists $k_0$ such that
$F(x_0)\in F_k(B_\varepsilon(x_0))$ for all $k\ge k_0$.
\begin{proof}
Let $y_0:=F(x_0)$.
Since $F^{-1}(y_0)$ is discrete, we can choose $r>0$ such that
$B_{2r}(x_0)\subseteq \cD$ and $B_{2r}(x_0)\cap F^{-1}(y_0)=\{x_0\}$.
After decreasing~$r$ if necessary, we can assume
$r\le\varepsilon$ for the given~$\varepsilon$. Write $\Omega:=B_r(x_0)$.
Then $y_0\notin F(\partial\Omega)$, where \(\partial \Omega\) is the border of \(\Omega\) in the sense of \cite{quarteroni2000numerical}, hence,
$\operatorname{dist}(y_0,F(\partial\Omega))/2=:\delta>0$ (note that
$F(\partial\Omega)$ is compact since $F$ is again continuous).
\parindent=.3cm
Choose $k'$ such that $y_k:=F_k(x_0)\in B_\delta(y_0)$
for all $k\ge k'$.
Choose $k_0\ge k'$ such that
$\|(F-F_k)\restr{\partial\Omega}\|_\infty<\delta$
for all $k\ge k_0$.
Then, for each of these~$k$, we have
\[\operatorname{dist}(y_0,F_k(\partial\Omega))\ge \operatorname{dist}
(y_0,F(\partial\Omega))-\|(F-F_k)\restr{\partial\Omega}\|_\infty
> 2\delta-\delta=\delta
\]
and, consequently, $B_\delta(y_0)\subseteq\R^n\setminus F_k(\partial\Omega)$.
Because of $y_k\in B_\delta(y_0)$, the points $y_0$ and~$y_k$ lie in the
same connected component of
$\R^n\setminus F_k(\partial\Omega)$. Therefore we have
\[
d(F_k,\Omega,y_0)=d(F_k,\Omega,y_k)
\]
where $d$ denotes the Brouwer degree (see e.g., \cite{ruzicka2004}).
The right-hand side of this equation is $\pm 1$ because
$F_k\restr{\overline{\Omega}}$ is an injective continuous map from a compact
set to a Hausdorff space, hence, a homeomorphism onto its image.
Thus, $d(F_k,\Omega,y_0)=\pm 1\ne0$ and, therefore, $y_0\in F_k(\Omega)$
for all $k\ge k_0$. The statement of the Lemma now follows from
$\Omega=B_r(x_0)\subseteq B_\varepsilon(x_0)$.
\end{proof}
\end{lemma}

\begin{proposition}
Let $\{F_k\}$ be defined as in Lemma \ref{lemm42}. Assume that
the preimage $F^{-1}(y)\subseteq \cD$ is discrete for every
$y\in\operatorname{im}(F)$.
Then $F$ is injective. 
\begin{proof} The Proposition follows immediately by contradiction. 
Suppose there were $x_1\ne x_2$ in~$\cD$ with $F(x_1)=F(x_2)=:y_0$.
Choose $\varepsilon>0$ small enough such that
$B_\varepsilon(x_1)$ and $B_\varepsilon(x_2)$
are disjoint subsets of~$\cD$.
Let $k_1, k_2$ be as in Lemma 4.3, that is, such that
$y_0\in F_k(B_\varepsilon(x_i))$ for all $k\ge k_i$, $i=1,2$.
Then $y_0\in F_k(B_\varepsilon(x_1))\cap F_k(B_\varepsilon(x_2))$
for every $k\ge\max\{k_1,k_2\}$, contradicting injectivity of the~$F_k$.
\end{proof}
\end{proposition}
Hence we obtain the following strengthening of Proposition \ref{SS+LIKQ}
\begin{corollary}
If $F$ is simply switched and stably coherently oriented in that all small perturbations have this property, then 
it is bijective.
\end{corollary}
%\newpage
The simply switched one-dimensional example $F(x) = x - |x- \zeta| + |x +
\zeta |$ is monotonically growing and thus coherently oriented if
$\zeta \leq 0$ but for $\zeta > 0$  it has a slope of $-1$ in a small
interval about the origin.  Hence, for the limiting case $\zeta =0$, where $F(x)
\equiv x$, we have coherent orientation, but that property is lost for
arbitrarily small $\zeta >0 $. Nevertheless, the function is of course
injective so that one might conjecture that
for simply switched PL functions openness already implies injectvity. 

However, that is not the case as one can see from the following instance
of the Rosette example.
\be  F(x) \equiv \begin{bmatrix}  |x_1| - |x_2| \\
 \textstyle\frac{1}{2} |x_1+x_2| - \textstyle\frac{1}{2} |x_1-x_2| \end{bmatrix}.  \label{schuethex} \ee 
It is simply switched and coherently oriented, but not injective since $F$
is even, so that $F(-x)=F(x)$. The LIKQ is violated since the four kinks
$\{x_1=0\},\{ x_2=0\}, \{x_1=x_2\}$ and $\{x_1 = -x_2\}$ all intersect at the origin.
Moreover, one can see that the perturbations
$$ F_\varepsilon (x)\;  \equiv \; \begin{bmatrix}  |x_1+\varepsilon| -|x_2+\varepsilon| \\
 \tfrac{1}{2} |x_1+x_2| - \tfrac{1}{2} |x_1-x_2| \end{bmatrix}   $$
are  no longer coherently oriented for $\varepsilon \neq 0$. More specifically,
for $\varepsilon >0$ we have the Jacobian
$$ F_\varepsilon^\prime \; = \;  \begin{bmatrix} 1 & -1 \\ 0 & -1
\end{bmatrix} \quad  \mbox{at}  \quad  x \; = \; \begin{bmatrix} x_1 \\ x_2 \end{bmatrix} \; = \;  \begin{bmatrix}
-\varepsilon/2 \\  - \varepsilon/4  \end{bmatrix} $$
whose determinant is $-1$ so that we do not have stable coherent orientation. 
 
\section{Generalized Newton Variants} 
If all elements of $ \partial^L \! F(x_*)$ are nonsingular  at some root $x_*\in F^{-1}(0)$, it follows from the celebrated theorem of Qi and Sun 
\cite{Qi:NonVN} that the full step iteration 
\begin{equation}
 x_{+} \; = \;  x - J_\sigma^{-1} F(x) ,\quad   \mbox{with} \quad J_\sigma \in  \partial^L F(x) \label{newton}
 \end{equation}
 converges from all $x_0$ sufficiently close to $x_*$. In fact, this result holds here trivially, since the iteration converges in one step from all points in the
open neighborhood
 $$ \Omega(x_*) \equiv  \{P_\sigma :  x_* \in \overline P_\sigma \}^\circ .$$  
 Of course, this means that all the combinatorial issues have already been resolved by the choice of $x_0$. 
 
Much more interesting is the question under which conditions the full step Newton method \eqref{newton}  
converges globally, i.e., from all initial points $x_0$. Using the mean value theorem of Clarke stated for example as Prop. $7$.$1$.$16$ in \cite{pang1992linear}, one can establish the following global convergence result.

\begin{proposition}[Full step convergence] $\quad $\label{propalt} \\ 
Given $x_\ast \in F^{-1}(0)$ the full step Newton method converges from all $x_0 \in \R^n$ in finitely many steps to $x_\ast$  if, with respect to some induced matrix norm, either 
of the following contractivity assumptions is satisfied % \(\forall \sigma \, \mathbf{open}\) and \(A \in \partial F(\R^n)\)
\begin{equation}
\| I -   J_\sigma^{-1}  J_{\tilde \sigma}   \|  < 1, \quad \mbox{for all} \quad \sigma, \tilde \sigma \, \mathbf{open}  ,    \label{errorred} 
\end{equation}
or
\begin{equation}
\| I -    J_{\tilde \sigma} J_{\sigma}^{-1}   \|  < 1, \quad \mbox{for all} \quad \sigma, \tilde \sigma \, \mathbf{open} .   \label{residred} 
\end{equation}
In either case the root $\{x_*\} = F^{-1}(0) $ is unique. 
\end{proposition} 
\begin{proof}  
 By the mean value theorem we derive from \eqref{newton} the solution error recurrence 
$$  x_{+} - x_\ast \; = \;  x- x_\ast - J_\sigma^{-1} A  ( x - x_*) \; = \;  \left [ I -  J_\sigma^{-1} A \right ]   ( x - x_*) $$
where for some $m \geq 1$ and $\lambda_i \in \R$ 
$$ A \; = \; \sum_{i=1}^m \lambda_i  \, J_{\sigma_i}  \quad \mbox{with} \quad \sum_{i=1}^m \lambda_i = 1 \quad \mbox{and} \quad  \lambda_i > 0. $$  
With a similar convex combination $\tilde A$ of limiting Jacobians  we find for the residual 
\begin{equation}
F(x_{+}) \;  =  \; F(x) -   \tilde A \,  J_{\sigma}^{-1} F(x) \; = \;  \left [ I -  \tilde  A \,  J_{\sigma}^{-1}\right ] F(x) . 
\end{equation} 
%where
%$$ A \; = \; \sum_{i=1}^m \lambda_i J_{\sigma_i}  \quad \mbox{with} \quad \sum_{i=1}^m \lambda_i = 1 ,  \; \lambda_i > 0. $$  
If we can ensure reduction of either norm $\|x-x_*\|$ or $\|F(x)\|$ by a fixed factor that implies at least linear convergence to a root. And then we eventually must reach an iterate $x$ such that 
$\partial F(x) \subset \partial F(x_*)$. In the next step we would get $x^+ = x_*$. 
By the triangle inequality and our assumption \eqref{errorred}  it follows  that %the matrix in brackets has a norm less than one so that $w$ and also $\Delta x$ must vanish, indeed.   
 \[ \left \lVert I- J_\sigma^{-1} \sum_{i=1}^m\lambda_i J_{\sigma_i } \right  \rVert =  \left \lVert \sum_{i=1}^m\lambda_i\left( I- J_\sigma^{-1}  J_{\sigma_i } \right)  \right \rVert \le \sum_{i=1}^m\lambda_i \left  \lVert I- J_\sigma^{-1}  J_{\sigma_i }  \right \rVert   \;  < \; 1 .\]
%so that $w$ and also $\Delta x$ must vanish, indeed.
Since the number of all Jacobians is finite, there is a global maximum of the term \eqref{errorred}, which bounds the reduction factor  $ \|x^+-x_*\|/\|x-x_*\|$. Similarly, \eqref{residred}  yields a bound less than $1$ on the ratio $\|F(x^+)\|/\|F(x)\|$.  This completes the proof.  
\end{proof}

The proposition deals with a special case of the general theory on  {\em nonsingular uniform Newton approximations} in the sense of \cite{pang1992linear}.
 Now we will look for sufficient conditions for the contractivity properties \eqref{errorred} or \eqref{residred} and thus global convergence of full step  Newton and injectivity of $F$ in terms of the abs-normal representation. 
  To obtain an explicit expression for the inverses $J_\sigma^{-1}$ we will assume that the matrix $J \in \R^{n\times n}$ representing the smooth part  of
 our function is nonsingular.  Should that a priori not be the case we can use the trivial identity 
 \be    v \; = \; ||v|+v |    - |v| , \quad \mbox{for} \quad v \in \R    \label{trivial} \ee
 to shift terms between the smooth and nonsmooth parts without changing the mapping $F$. However, for each modified entry we introduce two new switching variables and thus the abs-normal form and its various properties are significantly altered. Now we obtain the result.    
\begin{proposition} $\quad$ \\
 Assume that the abs-normal form of $F$ has an invertible smooth part $J$
and that
 $$\hat{\rho}\equiv \|J^{-1}Y\|_p\|Z\|_p<1-\|L\|_p  \; .$$ 
Then generalized Newton converges in finitely many iterations from any
$x_0$ to the then unique solution $x_*$ if
\be    \bar \rho \equiv\frac{2\hat \rho }{(1 - \hat \rho- \|L\|_p)(1- \|L\|_p)  }  \;  < 
\;  1 \; \;  \label{strongest}  \; .\ee
Moreover, the p-norms of both the solution error and the residual
are reduced by a factor no greater than $\bar \rho$ at each iteration.
\end{proposition} 

\begin{proof}
        It follows from \eqref{genjac} that %equationref 3
       \begin{align*} \| I-J^{-1}J_\sigma\|_p & \le \left . \|J^{-1}Y\|_p
\|(I-L\Sigma)^{-1}\|_p \|Z\|_p \; \leq  \;  \hat{\rho}\right /  (1-
\|L  \|_p).
        \end{align*}
        Hence we have by the Banach Pertubation Lemma that
        \[
        \|J^{-1}_\sigma J\|_p = \| [I-(I-J^{-1}J_\sigma)]^{-1} \|_p \le
 1 \left / [1-  \hat\rho \left / (1-\|L\|_p)  \right . ] \right . ,
        \]
        which immediately yields  for any pair of open signatures $\sigma,
\tilde \sigma$
 \be \| J_{\sigma}^{-1} J \|_p \; \leq \; \frac{1- \|L\|_p}{1 - \hat \rho-
\|L\|_p} .\label{invbound} \ee %\;  1 + \frac{\hat \rho}{1 - \rho} \; \le
% \eqref{divbound} is obtained
Furthermore we derive from  \eqref{genjac} that
\begin{eqnarray*}
  J^{-1} \left [J_{\tilde \sigma} -  J_{\sigma}\right ]  & = & J^{-1} Y
\left [ \tilde \Sigma (I-L \tilde \Sigma )^{-1} -  \Sigma (I-L \Sigma
)^{-1} \right] Z \\
  & = & J^{-1} Y \left [ (I-\tilde \Sigma L  )^{-1}  \tilde \Sigma- 
\Sigma (I-L \Sigma )^{-1} \right ]Z \\
   & = & J^{-1} Y (I-\tilde \Sigma L  )^{-1}  \left [  \tilde \Sigma (I-L
\Sigma ) -   (I- \tilde \Sigma L) \Sigma \right ]  (I-L  \Sigma  )^{-1}
 Z  \\
   & = & J^{-1} Y (I-\tilde \Sigma L  )^{-1}  \left [  \tilde \Sigma  - 
\Sigma  \right ]  (I-L  \Sigma  )^{-1}  Z .
 \end{eqnarray*} 
Now taking again norms and applying standard inequalities we find  
\be  \| J^{-1} \left (J_{\tilde \sigma} -  J_{\sigma}\right )  \|_p  \;
\leq \; \frac{2 \hat \rho }{(1 - \|L\|_p)^2} \; \, .  \label{divbound} \ee
By multiplication of  \eqref{invbound} and 
\eqref{divbound}, the last inequality ensures that both \eqref{errorred} and
\eqref{residred} are satisfied.
  \end{proof}

\subsection*{Piecewise Newton}
The conditions for the global convergence of full step Newton derived above are certainly rather strong and various globalizations like Ralph's 
path search have been proposed.   
On the other hand, it was observed in \cite{griewank2013stable} that coherent orientation implies that the fibres 
\begin{eqnarray} 
 [x_0]   \; \equiv \; \{ x  \in \R^n : F(x) = \lambda F(x_0), 0 < \lambda \in \R \} \label{fibre}
 \end{eqnarray} 
are, for almost all $x_0 \in \R^n$, bifurcation-free piecewise linear paths whose closure contains a root of $F$. The other singular fibres may have bifurcations, but there is always  a possibility to further reduce the residual towards a solution.   

The question how this piecewise Newton method is best implemented needs further investigation, but  numerical experiments are certainly encouraging \cite{piontkowski}.   
There is a key difference between this piecewise Newton and damped Newton in that piecewise Newton is not based on just any limiting Jacobian at 
the   current iterate, but on one that is indeed valid along the direction being taken. It cannot be guaranteed in the usual paradigm  that an oracle evaluates at any 
$x$ the residual $F(x)$ and some limiting Jacobian  $\partial^L F(x)$.    

We may summarize the results of this fourth and fifth section in the following graph of implications: 
 
\begin{center}
\! \! \! Contractivity $\Rightarrow$ Bijectivity $\; {\Longrightarrow} \; $ Openness $\Rightarrow $ Surjectivity $\qquad \qquad \qquad \qquad  \; \; \; \;  $\\
 \hspace*{4.6cm} $ \; ( \; \Longleftarrow \; $ if simply switched+stably coherently oriented ) 
\end{center}

The fact that the last two implications are  not reversible in general was already demonstrated in Section 2 on the 
Rosette example, which is not simply switched. The possibility of failure for full step Newton on bijective problems can be seen in the 
Rosette example \eqref{schuethex}. With a right-hand side $(1,-1)$ and a starting point $(2,1)$, Newton's method begins to cycle immediately.

\section{Schur complement and the complementary system} 
 
It turns out that we can eliminate $x$ when the smooth part $J$ is nonsingular .  
 \begin{lemma}\label{nonsingularschur} Provided that $\det(J) \neq 0$, we have 
 the Schur complement 
  $$S \equiv L- Z J^{-1}Y \; \in \; \R^{s \times s} $$
and in $P_\sigma$ it holds that
 $$ \det(J_\sigma) \; = \; \det(J)\, \det ( I - S \, \Sigma )  \; .$$
 Moreover,  if this determinant is nonzero, the inverse of $J_\sigma$ is given by
\be	J_\sigma^{-1} \;  = \; J^{-1} - J^{-1} \, Y\,\Sigma\, (I-S\, \Sigma)^{-1}Z\, J^{-1} .  \label{inv} \ee
\end{lemma} \label{SMW}
\begin{proof}
As Sylvester's determinant theorem states, that $\det (I+AB) = \det (I+BA)$, we have 
\begin{align*}
	\det(J_\sigma)/\det(J) & =    \det \left [ I  \! + \!  J^{-1} Y \Sigma(I\! - \! L \Sigma)^{-1}Z \right ]  
	 =    \det \left [I \!  + \! Z  J^{-1}\, Y\, \Sigma\, (I \! - \! L \Sigma)^{-1} \right ] \\
	&=   \det(I- L\Sigma)^{-1} \, \det\left (I - L\, \Sigma + Z\, J^{-1}\, Y\, \Sigma \right ) =  \det ( I - S \, \Sigma ) 
\end{align*} 
where we have used that the unitary lower triangular matrix $I- L\Sigma$ has determinant 1. 
 \end{proof}
 Whenever $J$ dominates the other three submatrices, things are not too difficult, as we will see below.  Notice that nonsingular linear transformations on the 
 independents $x$ and/or the dependents $y$ leave the Schur complement completely unchanged. At least for (generalized) Newton variants we could 
 therefore assume without loss of generality that $J=I$, although that does not seem to help all that much. 
 
  Rescaling the switching variables $z$ by a positive diagonal matrix $D$ would modify $Z$ to $DZ$, $Y$ to 
$YD^{-1}$ and replace $L$ by the similarity transformation $DLD^{-1}$, which is still strictly lower triangular.  One can choose $D$ such that the transformed $DLD^{-1}$ is arbitrarily small  in any one of the standard norms that are monotonic in the coordinates, but that may require a pretty wild scaling. More important is the Schur complement $S$, which would also be replaced by its similarity transformation $DSD^{-1}$. 

\subsection*{Conditions for coherent orientation}
The condition that $\det(I - S \Sigma)$ be positive for all switching
matrices $\Sigma$ is sufficient for coherent orientation of $F$ -- a property that would characterize $S$ as $nonexpansive$ in the sense of Theorem 6.1.3 in \cite{neumaier1990interval}. In Theorem 2.3 of \cite{rump1997theorems}  Rump gave several equivalent properties, one of
which is that the {\em sign real spectral radius} 
$$   \rho_0^s(S) \; \equiv \; \max \left \{ \rho_0(\Sigma \,S ) : {\Sigma \in \mathbf{diag} \{-1,1\}^n}  \right \}  $$
is less than $1$. Here $\rho_0(S) \leq \rho(S)$ denotes the {\em real spectral radius} of a square matrix, i.e.,  the largest modulus of any real eigenvalue of $S \in \R^{n \times n}$. The complex eigenvalues are ignored in this maximization, which makes $\rho_0(S)$ highly discontinuous with respect to $S$. Remarkably,  $\rho_0^s(S)$ is again continuous in the entries of $S$ and it  vanishes exactly when $S$ is permuted strictly triangular. This  is true for the leading part $L$ of our Schur complement so that we must have  $\rho_0^s(S)< 1 $ when the additional term $Y J^{-1} Z$ is sufficiently small. 
In general,  deciding whether  $\rho_0^s(S)$ lies below  a given bound is an NP hard problem. Rump
also showed that the following property is sufficient, but not necessary for $\rho_0^s(S)< 1 $ and, thus,
coherent orientation.

 \begin{definition}
 An abs-normal form of F is called smoothly dominant if
 \[ \rho\equiv \| DSD^{-1} \|_p<1
  \]
for some  p-matrix norm and some positive diagonal scaling $D$.
\end{definition}
\noindent
This condition was already used by Bokhoven in his dissertation \cite{van1981piecewise}. Similarly, Mangasarian and Meyer \cite{mama} wrote their absolute value equation 
$      A \, x - | x| \; = \; b    $
in terms of the inverse $A= S^{-1}$. 

 Assuming smooth dominance of $A^{-1}$ for the special choice  $p=2$ and $D=I$ they showed unique solvability of the AVE. This can be shown directly using the contractivity of what Bokhoven and his followers call the {\em modulus algorithm} as discussed below. First we will show that  coherent orientation may be present even when all \(p\)-norms are substantially greater than \(1\), i.e., when the PL system is far 
from being smoothly dominant.    
 
\begin{lemma}
There are matrices $S_n \in \R^{n\times n}$ with signed real spectral radius   $\rho_0^s(S_n) \leq 0.9 $ for which all $p$ norms 
$\|D_n^{-1} S_n D_n\|_p$ with arbitrary diagonal scalings $D_n>0$ are greater than $1$, for $n \ge 3$ and furthermore   
$\lim_{n} \|D_n^{-1} S_n D_n\|_p = \infty .$   
\end{lemma} 
\begin{proof}
Dropping the subscript $n$ and abbreviating $ e  \equiv (1 \ldots 1)^\top \in \R^n, \;  I \in \R^{n\times n} $ we consider   Rump's example
\begin{equation}
  S =  \tfrac{9}{10}\cdot(\sign(j-i))_{i,j=1\dots n}\in\R^{n\times n}     \quad \mbox{with} \quad  |S| = \tfrac 9{10}\left(e\,e^\top - I\right)  \; .    \label{eqn:RumpI}
\end{equation}  
Since,  for any $D = \diag(d)\in\R^{n\times n} $ with (\(d > 0\), componentwise) 
\[ \lVert DSD^{-1} \rVert_\infty = \lVert D\,|S| D^{-1} \rVert_\infty\; , \]
we obtain 
\begin{align*}
  \tfrac {10}9\lVert D\, S \,  D^{-1} \rVert_\infty &= \lVert D(ee^\top - I_n) D^{-1} \rVert_\infty = \lVert D\, ee^\top D^{-1} - I_n \rVert_\infty \\
  &\ge \left\lvert \lVert D\, ee^\top D^{-1} \rVert_\infty - \lVert I\rVert_\infty \right\rvert = \left\lvert \max_{ 1 \leq j \leq n} \sum_{i=1}^n\tfrac{d_i}{d_j} - 1 \right\rvert \; .
\end{align*}
% \hspace*{20pt}(note that \(D\,ee^\top D^{-1} = \left[d_i/d_j\right]_{i,j = 1,\dots ,n}\))
By elementary arguments one can see that the expression on  the RHS attains its minimal value $n-1$ when all \(d_j\) are equal so that
\[ \lVert D\, S \,  D^{-1} \rVert_\infty \ge \tfrac 9{10}(n-1). \]
\noindent Now let \(\hat x \in \R^n\) be the unit vector \( \lVert \hat x\rVert_\infty = 1\) that maximizes the infinity norm \(\lVert DSD^{-1} \hat x\rVert_\infty\), such that
\begin{align*}
  \lVert DSD^{-1} \rVert_p &= \max_{\lVert x \rVert_p = 1} \lVert DSD^{-1}x \rVert_p \\ 
  &\ge \frac{\lVert DSD^{-1}\hat x \rVert_p}{\lVert \hat x\rVert_p} \ge \frac{\lVert DSD^{-1}\hat x \rVert_\infty}{\lVert\hat x\rVert_p}  \; . 
\end{align*}
Finally this yields by the equivalence of the vector norms \( \lVert \hat x\rVert_p \le n^{\frac 1p}\lVert \hat x\rVert_\infty = n^{\frac 1p} \) % 
\[ \frac{\lVert DSD^{-1}\hat x \rVert_\infty}{\lVert\hat x\rVert_p}  \;  \ge \;  \frac{n-1}{n^{\frac 1p}} \; \;  \xrightarrow{n\to\infty} \; \;  \infty\; . \]
On the other hand,  we know from \cite{rump1997theorems}  that the sign real spectral radius satisfies 
$\rho_0^{s}(S) = 0.9 < 1$ so that we have coherent orientation of $F$ as asserted. 
\end{proof} 

To see that smooth dominance can also arise when $\rho(|S|) > 1$ let us consider the $2 \times 2$  matrix 
$$ S \; = \;  R(\tfrac{\pi}{2}) \; = \;  \frac{0.9}{\sqrt{2} } \begin{pmatrix}  1 & -1 \\ 1 &  1 \end{pmatrix} \; .$$
 It represents a rotation by $\pi/2$ followed by a contraction  by $0.9$. Then we have 
 $$   \|S\|_2    = 0.9 \; < \;  1 \; < \; {0.9}\,  {\sqrt{2}}    \; = \; \rho( |S|) \; .$$
As a more  interesting example for smooth dominance let us consider a problem
$$ T x + \max(x,0) = b ,   \quad \mbox{where} \quad T  \; \succ\; 0   $$
is  symmetric positive definite, which is the stronger assumption used in  \cite{brugnano2008iterative}. (The $\max$ is meant componentwise.)   
Rewriting this problem in abs-normal form using  $\max(x,0) \equiv (x + |x|)/2 $ we obtain
$$ z = x \quad \mbox{and} \quad y = -b + (T+I/2)x + |z|/2  \; .$$
This corresponds to $c=0, Z = I, L = 0, J = T+I/2, Y = I/2$ and yields the Schur complement 
$ S \; = \; 0 - (T+I/2)^{-1}/2 = - (I + 2 \, T)^{-1}  \prec 0$. It is negative definite  with spectral radius below $1$. 
Hence, we have smooth dominance as $ \|D S D^{-1} \|_2 \; < \; 1$ for  $D=I$. We have verified that the fixed point iteration suggested in \eqref{18a}  below converges when $T$ is the usual second order divided difference stencil. However, it does so very slowly and applying the generalized Newton iteration \eqref{newton} and equivalently \eqref{newtonz}, also  advocated in \cite{brugnano2008iterative} turns out to be much more effective.

An even stronger condition for smooth dominance and thus coherent orientation  follows  from the well known result of Perron-Frobenius.
\begin{lemma}{Perron-Frobenius scaling} \label{perronfrob} \\
Suppose that $S$ and hence its componentwise modulus $|S|$ is not  permuted block-triangular. 
Then the spectral radius $\rho(|S|)$ is positive and the corresponding eigenvector $d \in \R^n$ is strictly positive such that for 
$D= \mathbf{diag}(d)$ and $e = (1\ldots 1) \in \R^s$ 
$$  D^{-1} S d  \; \equiv  \; D^{-1} S D \, e  = \rho(|S|) e  \; \implies \; \|D^{-1} S D \|_\infty  \; = \; \rho(|S|) \; .$$
If $\rho(|S|)=0$, the norm  $\|D^{-1} S D \|_\infty$ can be made arbitrarily small.  
\end{lemma}
\begin{proof}
It is well known that all components of the eigenvector $d$ are positive if the corresponding eigenvalue $\rho(|S|)$ is nonzero. Then we find immediately that 
$e$ is the eigenvector associated with the largest eigenvalue of $|\tilde S|$ for $\tilde S = D^{-1} S D$ , which in turn shows that $\|\tilde S \|_\infty = \|| \tilde S| \|_\infty$ has the same value. If $\rho(|S|)=0$, we can add $\varepsilon \, e\, e^\top$ to $|S|$ and apply the first observation to establish the second.     
\end{proof}

According to the lemma, {\em absolute contractivity}, i.e. $\rho(|S|) <1$, implies smooth dominance in the infinity norm. Moreover, we may always similarity transform $S$ by some diagonal 
$D>0$ such that all rows of  $\tilde S \equiv D^{-1}S D$ have the same $l_1$ norm equaling $\rho(|S|)= \rho(|\tilde S|)$. We will call this process {\em equilibration}. This may not work if  $S$ is reducible in that it is permuted block triangular, which can for example be tested by the algorithm given in \cite{duff1986direct}. 
In the reducible case the complementary system discussed below can be decomposed into several subsystems, to which our solution techniques can be applied successively. Consequently, we may assume from now on without loss of generality that the sparsity pattern of $S$ is irreducible, which also implies $\rho_0^s(S) > 0$.       
Alternatively, we can scale by the left Perron-Frobenius vector $\tilde d$ of $|S|$ to achieve $  \|\tilde D^{-1} S \tilde D \|_1 \, = \,  \rho(|S|)$ for 
$\tilde D = \mathbf{diag}(\tilde d)$, but that appears to be of little help here.

\subsection*{The complementary system}
We will assume throughout that $J$ is nonsingular, hence, that $S$ is well defined and that a suitable scaling was applied to make some norm 
$\|S\|_p$ small, if not necessarily less than one.  So far we have looked at \eqref{absnormal} as a system that defines a unique $z \in \R^s$ and thus a corresponding $y$  for each $x \in \R^n$ via the first set of $s$ triangular equations. Now suppose we have given a fixed target value $y$, which we can subsume into $b$, and compute for each $z$ the corresponding value 
\begin{equation}
 x \; = \; x(z)\; \equiv   - J^{-1}(b + Y |z|)    \label{ztox}  \; . 
 \end{equation} 
Substituting this result into the first equation we obtain for $z$ the PL system 
\be   H(z) \; \equiv \;   z  - L | z|   + Z J^{-1} Y |z| \; = \; (I - S \Sigma ) z  \; = \;  \hat c \; \equiv \;   c  - ZJ^{-1}b \label{unfolded}   \; .\ee
Provided $S$ has the inverse $A$ we may write equivalently 
\be H(z) \; = \; z - S |z| \; = \hat c    \quad \iff \quad A \, z  - |z| \; = \; \hat b \equiv A\, \hat c   \; .  \ee
Here the right hand side represents the absolute value equation of Mangasarian and Mayer \cite{mama}. They make the interesting observation that if $A$ is sufficiently small then only strictly negative rights hand sides $\hat b$ lead to solutions. Moreover, according to their Proposition 6 these inverse image sets attain all possible $2^n$ sign combinations, as is   obvious for the limiting case $-|z| = \hat b$, where $A$ vanishes.  Intuitively it would seem that such complete domination of the smooth part by the nonsmooth part makes little sense in a realistic model. Correspondingly,  Mangasarin and Mayer also consider the situation where $A$ is sufficiently large or in our formulation $S$ is sufficiently small, e.g. in the sense of smooth dominance.      
 
Note that the generalized Jacobians $(I - S\,\Sigma)$ of the complementary vector function $H(z)$ all have the same determinant sign if and only if $\rho_0^s(S) <1$, 
which we encountered as a sufficient condition for the coherent orientation of $F$. Generally, $F(x)$ must be coherently oriented if this is true for $H(z)$, but the converse implication is usually not true. The reason is that while all possible sign combinations of $z$ arise in the domain $\R^s$ of $z$, the switching variables $z = z(x)$ are typically restricted to a Lipschitzian submanifold in $\R^s$ as $x$ ranges over $\R^n$.    

Conversely, for any given $z$ solving the lower part of  \eqref{absnormal} for $x$ yields the corresponding value 
%and then its upper part for $z^+$ is equivalent to the relation 
\begin{equation}
 z \; = \; z(x)\; \equiv   G^{-1}(c + Z x)  \quad \mbox{with}  \quad G(z) \; \equiv \; z - L |z| .    \label{xtoz}  
 \end{equation} 
As stated by Lemma \ref{perronfrob}  we can make any $p$-norm  $\|L\|_p$ of the strictly lower triangular matrix $L$ as small as possible and in particular smaller than $1$. Then the existence of $G^{-1}$ follows not only from the triangularity of $L$ but also the Banach fixed point theorem. Now we can observe that solutions of the original problem OPL and the complementary problem CPL correspond to each other. 
\begin{lemma}[One-to-one solution correspondence\label{1-1}] ~\\ 
Under our general assumptions with $\det(J) \neq 0$ a point   \(x_\ast \in \R^n\) is a solution of the OPL \(F(x) = 0\) 
if and only if it is a fixed point of $x(z(x))$, which is in turn equivalent to $z_*=z(x_*)$ being a fixed point of $z(x(z))$ and equivalently a solution of the CPL
 \(H(z)  =\hat c\) .  
\end{lemma}
\begin{proof}
We have the equivalences $F(x)=0$  
\begin{eqnarray*}
   & \iff & x  = -J^{-1}[b +Y |z| ]   \quad \mbox{with} \quad    z = c + Z x + L |z| \\
& \iff  &  \;  x = -J^{-1}[b +Y |z| ]    \quad \mbox{with} \quad   G(z) = c + Z x \\
&   \iff &   \;  x = -J^{-1}[b +Y \left | G^{-1}(c + Z x )\right | ] \\
& \iff   & \;  x = x(z(x))    \;  \iff \;   z = z(x(z))  \\
& \iff & \; z =  G^{-1}(c + Z x)  \quad \mbox{with}  \quad x = - J^{-1}(b + Y |z|)  \\
& \iff & \; z =  G^{-1}(c -  Z  J^{-1}(b + Y |z|))  \\
& \iff & \; G(z) =  c -  Z  J^{-1}(b + Y |z|)  \\
& \iff  &\; z - L |z|  =  c -  Z  J^{-1}(b + Y |z|)  
\end{eqnarray*}
which is equivalent to  $ H(z) = \hat c  $ defined in \eqref{unfolded} as asserted. 
 \end{proof} 
We may interpret $H(z)$ as a simply switched  PL function in abs-normal form with $z \equiv x, Z =I =J, L=0$, and $Y= -S$. The Schur complement is then again 
$0- I\ I^{-1}(-S) = S$, which was to be expected. Since the LIKQ condition is  satisfied, the complementary function $H(z)$ is always bijective if and only if it is 
open, which happens exactly when $\rho_0^s(S) <1$.      

\section{Solving the complementary system CPL}   
In view of Lemma \ref{1-1} we can hope that the largely equivalent fixed point iterations $x^+=x(z(x))$ and $z^+=z(x(z))$ defined by \eqref{xtoz} and \eqref{ztox} lead to convergence.  
As it turns out it is a little easier to establish convergence of the coupled iteration with respect to the $z$-component and the $x$-component must then converge to its own fixed point by continuity. 
\begin{proposition} The Block Seidel iteration $z^+ \; = \; z(x(z)) $  converges from all $z_0$ to the unique fixed point $z_*$ if in some p-norm 
$$  \|S-L\|_p + \|L\|_p  \; < \; 1 \; .  $$ % \quad \mbox{which} \m$$
Moreover, the corresponding  $x_* = - J^{-1}(b+Y|z_*|)$ is the unique root of $F(x)=0$. 
\end{proposition}
\begin{proof} 
Since for any pair $z,\bar z \in \R^s$ by the inverse triangle inequality 
$$ \|G(z)-G(\bar z)\|_p  \; = \;  \| (z-\bar z) - L (|z|-|\bar z|)  \|_p \; %\geq \;   \|z-y\|_p -\| - \|L\|_p \||z|-|y|\|_p
  \geq  \|z-\bar z\|_p (1 -  \|L\|_p ) $$
the inverse $G^{-1}$ has the Lipschitz constant $1/ (1 -  \|L\|_p |) $. The Lipschitz constant of the map $R(z) \equiv \hat c  - Z J^{-1}Y |z| $ is simply $\| Z \, J^{-1} \, Y\|_p$, which can be expressed in terms of the Schur complement as $\|S-L\|_p$. Using the multiplicativity of Lipschitz constants  we derive for the fixed point iteration $z(x(z))=G^{-1}\circ R(z)$ 
$$ \sup_{z \neq \bar z}  \frac{\|G^{-1}\circ R(z)- G^{-1}\circ R(y)\|_p}{\|\bar z-z\|_p}
 \; \leq \;  \frac{\| Z \, J^{-1} \, Y\|_p}{ 1 -  \|L\|_p } \; = \;  \frac{\|S - L\|_p}{ 1 -  \|L\|_p } \; . $$
Since the last upper bound is less than $1$ exactly when the assumption of the proposition is satisfied, convergence follows again by Banach's fixed point theorem. The last assertion holds by substitution of $(x_*,z_*)$ into \eqref{absnormal}. 
\end{proof} 
\subsection*{Modulus Algorithm}
It follows immediately from the triangle inequality that the fixed point iteration can only be guaranteed to converge when the problem is at least smoothly dominant in that 
$\|S\|_p< 1$. Under that somewhat weaker condition one may apply the simpler fixed point iteration
\be  z^{+} \; = \; \hat H (z) \; \equiv \hat c + S |z| \label{18a}  \; . \ee
Here no triangular substitution process is needed and $S$ may or may not be formed explicitly. If not, we have to just solve one linear system in $J$ at each iteration and multiply vectors by the matrices $Y, Z$ and $L$.  A lack of smooth dominance may then only be discovered by nonconvergence. 
 This simple fixed point iteration was introduced as modulus algorithm in
Theorem $10$ on page $72$ of \cite{van1981piecewise} and spawned the development of many
variations ( see e.g. \cite{hadj} and citations). 
 We restate the basic convergence result. 
\begin{proposition}
If the abs-normal form of $F$ is smoothly dominant in that $\rho = \|S\|_p < 1$, then the iteration \eqref{18a} converges for all $\hat c$ 
 from any $z_0$ to the unique solution $z_* = H^{-1}(\hat c)$. 
 \end{proposition}
 \begin{proof} 
To prove contractivity of $\hat H$ on $\R^s$ we note that 
$$  \|\hat H(z)-\hat H(\tilde z)\|_p \; = \; \|  S (|z| -|\tilde z|) \|_p \; \leq \; \|  S \|_p \| |z| -|\tilde z| \|_p \; \leq \; \rho \|| z -\tilde z| \|_p \; =  \; \rho \| z -\tilde z \|_p . $$
Thus, the Banach fixed point theorem ensures  linear convergence to a unique root with monotonically declining error norm  $\|z - z_*\|_p$.  
 \end{proof}
To verify that coherent orientation is not sufficient for the fixed point iteration to converge we applied it to the example  \(S = S_n\) from (\ref{eqn:RumpI}) for $n=1000$
with  \(\mathbf{c} = (\sin(i))_{i=1\dots n}\) and \( z_0 = \mathbf{0} \in \R^n\). Then \( z^+ = \hat c + S\lvert z\rvert \) diverges immediately. 
Whether there can be convergence of the fixed point iteration from generic starting points without smooth dominance is not yet clear.

\subsection*{Generalized Newton on CPL}
The convergence of the fixed point iterations is quite reliable, but may be asymptomatically rather slow.  
In particular, neither fixed point iteration promises finite convergence, so we wish to again examine Newton variants.   
Applying the generalized Newton method to $H(z)=\hat c$ we obtain the recurrence 
\begin{equation}
 z^{+} \; =  z - A^{-1} (H(z) - \hat c) , \quad   \mbox{with} \quad A \in  \partial^L H(z)\label{newtonz} \; .
\end{equation}
Since all $A$ now have the simple form $I - S\,  \Sigma$, we obtain as a specialization of Proposition \ref{propalt}  
  \begin{proposition}\label{73}
 If the abs-normal form of $F$ is smoothly dominant such that $\rho = \|S\|_p < 1/3$, then the iteration \eqref{newtonz} converges for all $\hat c$ 
 in finitely many iterations   from any $z_0$ to the unique solution $z_* = H^{-1}(0)$. Moreover,  the p-norms of both $z-z_*$ as well as $H(z)-\hat c$ are monotonically reduced
 \end{proposition}
 \begin{proof}
 We simply need to bound the norm of $J_\sigma^{-1}(J_\sigma - J_{\tilde{\sigma}})$ according to \\
 $ \|(I- S \Sigma)^{-1} S (\Sigma - \tilde \Sigma) \|_p \leq \|(I- S \Sigma)^{-1}\|_p \|S\|_p \| \Sigma - \tilde \Sigma\|_p \leq 2 \rho/(1- \rho) < 1$. 
 \end{proof}
 Since $H$ is simply switched all generalized Jacobians in $\partial^C \! H(z)$ are  by Proposition \ref{genjacprop}  of the form $(I- S \, \hat \Sigma) $ with $\hat \Sigma =  \mathbf{diag}(\hat \sigma)$ for some $\hat \sigma \in  [-1,1]^s$. Then we have still  $|z|   = \hat \Sigma z$, which is equivalent to $\hat \sigma_i z_i \geq 0$ for $i= 1\ldots s$. 
 Hence the previous proposition applies also if in \eqref{newtonz} the matrix $A$ is chosen as an arbitrary element of the set $\partial^C \! H(\tilde z)$, which contains only nonsingularar matrices. 
 
 \noindent    
Substituting $ A = I - S\, \hat \Sigma $ into  \eqref{newtonz} one finds that
 \be  z^+ \; = \; (I-S  \, \hat \Sigma)^{-1}  c  \quad \mbox{with} \quad   |z| = \hat \Sigma \,  z  \label{newtonz2} \ee
 which means that the generalized Newton iterate $z^+$  depends only on $\sign(z)$. 
 If $z$ is definite in that it contains no zero components we must have $\hat \sigma = \sign(z)$ and $z^+$ is uniquely determined.   
 For generic $\hat c$ the $2^n$ possible images $z^+$ defined by a definite $z$ will also be definite, as we will assume for the time being. 
   
 Then we may interpret the Newton iteration as a 
 finite automaton with the transition function $\sigma^+ = N(\sigma)$ on the state set ${\cal V} \equiv \{-1,1\}^n \equiv \{-,+\}^n$. 
  We can also regard the $\cal V$ as the vertex set  of a directed graph with the edges $(\sigma, N(\sigma))$. An example with $n=3$  is shown in Fig.   \ref{newtgraph}. It is a special case of an example used later in Proposition 7.7.
 
 \begin{figure}[htbp]
 \begin{center}
\resizebox{.66\linewidth}{!}{
\includegraphics{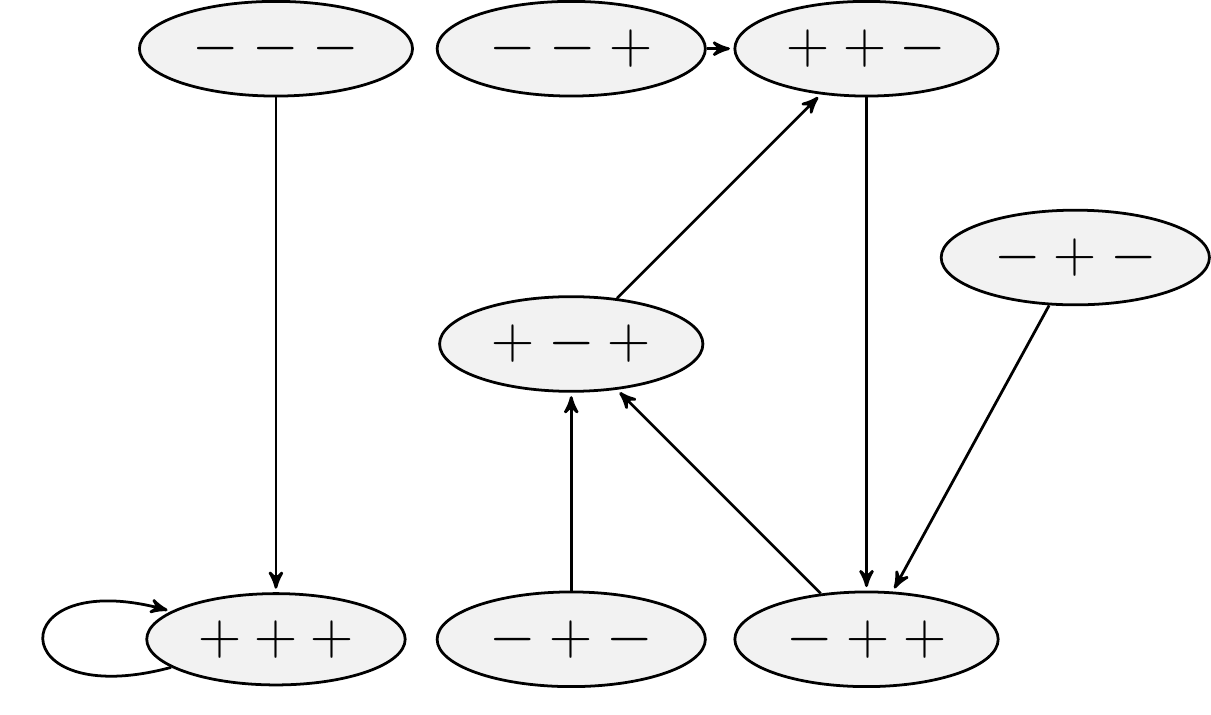}
}
\caption{Transition Graph of Newton's method on complementary system}
\label{newtgraph}
\end{center}
\end{figure}
\noindent
Since all vertices $\sigma$ in the directed graph $\cal G$ have a unique outgoing arc $(\sigma, N(\sigma))$ its structure is rather simple. Depending on the initial point one Newton's method either converges in finitely many steps or begins to cycle.  
\begin{proposition}\label{74}
Each connected component of the transition graph $\cal G$ contains a cycle of length greater than \(1\) or a unique fixed point, which is a cycle of \mbox{length \(1\)}.  
\end{proposition}
\begin{proof}  
From any initial $\sigma_0$ the sequence of iterations $\sigma_k = N^k(\sigma_0)$ stays in the connected component of $\sigma_0$ and must reach a fixed point or begin to cycle.
Let $C(\sigma_0)$ denote the set of vertices that are touched infinitely often by this sequence. Let $Prec(C(\sigma_0))$ denote the set of all $\sigma \in {\cal G}$ with 
$C(\sigma)=C(\sigma_0)$. We now have to exclude that  the connected subgraph $Prec(C(\sigma_0))$ has outgoing or incoming edges. There can be no incoming edges because repeatedly 
applying $N$ to their origins would also lead to $C(\sigma_0)$. Also there can be no outgoing edges because their origins would lead to a cycle or fixpoint outside 
$Prec(C(\sigma_0))$. This completes the proof. 
\end{proof} 
While the condition $\rho = \|S\|_p < 1/3$ used in Prop. \ref{73} excludes cycling it does seem rather strong.  
Alternatively, we may impose the condition $\rho(|S|) < 1/2$, which allows us to  prove  finite termination
 and even limit the computational effort to $n^3/3$ fused multiply adds. %In fact,  the method simply amounts to a signed version of Gaussian elimination.    
   \begin{proposition}\label{75}
 Let the Schur complement $|S|$ be absolutely contractive  with $\rho = \rho(|S|) < 1/2$ or  $\rho = 1/2$ and $S$ irreducible. 
 Then for all $\hat c$  any iteration \eqref{newtonz2} converges  in at most $s$ iterations from any $z_0$ to the unique solution $z_* = H^{-1}(c)$. 
 \end{proposition}
 \begin{proof}
 After equilibration by the Perron-Frobenius vector we may assume without loss of generality that  $\rho = \rho(|S|)=\lVert S \rVert_\infty \leq 1/2$.   
 For notational simplicity we drop the superscript $\hat {\;}$ and write $\sigma \in [-1,1]^s$ and $\Sigma = \diag(\sigma)$ with the only restriction  that at the current iterate $z$ we have $|z| = \Sigma z$.  
 
 The argument below will be based on the fact that for $\lVert S\lVert_{\infty}\leq 1/2$ with $S$ irreducible, the inverse $(I-S\Sigma^{-1})$ is strictly diagonally dominant with a positive diagonal. We will prove this statement for $\lVert S\lVert_{\infty}< 1/2$. The limiting case requires a more extensive reasoning, for which we refer to Lemma 4.2. in \cite{radons2014sign}. 
 
 Since $\lVert S\Sigma\lVert_{\infty} \leq \lVert S\lVert_{\infty} \lVert\Sigma\lVert_{\infty} \leq \lVert S\lVert_{\infty}$ it suffices to consider the case $\Sigma = I$: We have 
 $\lVert S^k\lVert_{\infty}\le\lVert S\lVert_{\infty}^k<\frac1{2^k}$ which implies
$\lim_{k\to\infty}S^k=0$. Hence we can express
$(I-S)^{-1}$ via the Neumann series
$$
A^{-1}=\sum_{k=0}^\infty(I-A)^k=\sum_{k=0}^\infty S^k = I+\sum_{k=1}^\infty
S^k.
$$
The inequality $\|\sum_{k=1}^\infty S^k\|_\infty\le
\sum_{k=1}^\infty\|S\|^k_\infty<\sum_{k=1}^\infty\frac1{2^k}=1$
already ensures strict diagonal dominance for $(I-S)^{-1}$.

 %By the Banach Lemma we have
 %$$ \| (I- S \, \Sigma)^{-1} \|_\infty \; \leq \; 1/(1- \rho)  \; < \; 2 \quad \implies  \quad  \|z^+\|_\infty < 2 \, \|c\|_\infty .$$    
Now we perform symmetric pivoting by reordering the equations and the  components of $z$ such that the first component $c_1$ of the permuted vector $c$ is its largest, i.e., 
 $|c_1|= \|c\|_\infty$. 
 Note that reorderings  of the equations and variables do not affect the generalized Newton iteration at all.  If $c_1=0$ we must have that $c=0$ and thus $z^+=0$ is obtained as the correct solution from any $z$ in one step.   
 Otherwise we have for the first component of the defining equation 
 $$     | z^+_1 - c_1 |   = | e_1^\top S \,  \Sigma \, z^+ |\;  \leq \;   \| e_1S\|_1 \|z^+\|_\infty  \; \leq  \rho \, 2   |c_1|  <  |c_1|   \; .$$
This ensures that the sign of the first component  $z^+_1$ is the same as that of  $\sigma_1^*  \equiv \sign(c_1) \neq 0$  and we have the crucial identity 
$|z_1^+| = \sigma_1^* z_1^+ . $   
This will remain true over all subsequent iterations since we have so far not imposed any assumptions on the step defining $\sigma$ whatsoever.
Hence we may assume that from the second iteration onwards already  $\sigma_1 = \sigma_1^* $ 
and thus also  $|z_1^+| = \sigma_1 z_1^+$.  
This relation allows us to rewrite the first equation and express it as a linear combination of the other $z_j^+$, namely 
$$ z_1^+(1- \sigma_1^* s_{11})  \; = \;  c_1  + \sum_{j=2}^s s_{1j}\sigma_j z_j^+  \; \implies \;    \sigma_1 z_1^+ = \frac{c_1}{\sigma_1^* - s_{11}}
+ \sum_{j=2}^s \frac{s_{1j} \sigma_j z_j^+}{\sigma_1^* - s_{11}}  \; .$$
Substituting this relation into the other equations, which corresponds to one step of Gaussian elimination, we obtain for $i=2\ldots s$
$$ z_i^+   \; = \; c_i +   \frac{s_{i1} c_1}{\sigma_1^* - s_{11}}   + 
\sum_{j=2}^s \left [s_{ij} + \frac{s_{i1} s_{1j}}{\sigma_1^*-  s_{11}}
 \right] \sigma_j z_j^+    
\; \equiv \;  {\tilde c}_i + \sum_{j=2}^s {\tilde s}_{ij} \, \sigma_j z_j^+  \; .
$$  
Hence we see that the other components $z^+_i$ for $i=2 \ldots s$ are equivalent to the ones that would be obtained on the reduced system with the same restricting for 
picking $\sigma_i$, namely $\sigma_i \, z_i = |z_i|$. The implicitly reduced matrix 
$\tilde S \equiv (\tilde s_{ij})^{i=2\ldots s}_{j=2\ldots s}$ satisfies $\| \tilde S \|_\infty \leq \rho = \| S \|_\infty$ since, for each $i > 1$,
$$  \sum_{j=2}^s |\tilde s_{ij}|  \leq   \sum_{j=2}^s |s_{ij}| +  \frac{|s_{i1}|}{1-\sigma_1^* s_{11}} \sum_{j=2}^s |s_{1j}|  \leq \rho - |s_{i1}| +    
\frac{|s_{i1}| (\rho- |s_{11}| )}{1-\sigma_1^*s_{11}}  \leq \rho - \frac{|s_{i1}|}{2}   \leq \rho  \; .$$
Thus we can repeat the argument and after the second iteration the sign of the $z_i^+$ corresponding to the maximal value of $|\tilde c_i|$ will be correct and nonzero.  
Moreover,  the others will be equivalent to those obtained under the same rule on a doubly reduced $(s-2)\times(s-2)$ system. Eventually the last element  of $z$ will be correctly identified and then all other 
components of the $s$-th generalized Newton iterate must be correct as well.  
 \end{proof}
 \subsection*{Signed Gaussian Elimination} 
The system reduction in the proof of the previous theorem depends only on the sign $\sigma_i^* = \sign(c_i)$ of an absolutely largest RHS component $c_i$ but not the initial guess 
of $z_i$ and a compatible $\sigma_i$.  As we have elaborated on in \cite{radons2014sign}, it can be applied directly to generate a signed Gaussian elimination procedure. Thus we obtain the following corollary:   
 
\begin{corollary}\label{76} If $\rho(|S|) <  \tfrac{1}{2}$ or $\rho(|S|) = \tfrac{1}{2}$ and $S$ irreducible  the unique solution of  the complementary system 
$z = S\,|z|+c$ can be computed by signed Gaussian elimination in at most $s^3/3$ fused multiply add operations plus $O(s)$ divisions. 
\end{corollary}
% \begin{remark}
%  Note that, while the proofs for their range of convergence or correctness respectively stem from a similar argument, the generalized Newton iteration and the signed Gaussian elimination are by no means equivalent methods, as we see below.  
% \end{remark}
Propositions \ref{73} and \ref{75} ensure the finite convergence of the generalized Newton method under the conditions $\rho = \|S\|_p < 1/3$ and 
$\rho = \rho(|S|) \leq 1/2$, respectively. Obviously, the second condition does not imply the former, but the converse does also not hold so that there are problems where only one but not both theorems apply. To demonstrate this we consider the example 
$$       S = 0.3  \, [ I - ee^\top/9 ]     \; \in \; \R^9   \quad \mbox{with} \quad e =(1)_{1\ldots 9} .   \;    $$
Here $S$ is  a scaled elementary reflector so that   $\rho = \|S\|_2 = 0.3 \cdot 1 < 1/3$. However, one can easily check that $\rho(|S|) =  0.3 \cdot 16/9 = 1.6/3 > 0.5 $ so that Proposition \ref{73} applies, but neither Proposition \ref{75} nor its Corollary \ref{76}. 
 
 \subsection*{Divergence of the generalized Newton on Cyclic Example} 
Another question that arises is whether the bound $1/2$ imposed on $\rho = \rho(|S|)$ in Proposition \ref{75} and its corollary could not be weakened. The answer is that for  $s$ of any significant size the bound may only be raised a minute amount above $1/2$ without opening the possibility of divergence. More specifically, we have the following family of counter examples, whose instance for $s=3$ was already depicted in Figure \ref{newtgraph}.       
 
\begin{proposition}
        For \(s > 2\) set  $\hat c = (1)_{1 \ldots s}$ and define  \(S\in\R^{s\times
s}\) as the cyclic T\"oplitz matrix
        \[ S = \begin{bmatrix} \mathbf{0} & a \\ a\,  I_{s-1} & \mathbf{0}
\end{bmatrix} .\]
        Then, if  $a \in \R$ satisfies
        \[ \frac 12 + \frac{1}{2^s} \; \leq \; a  \; \leq  \; \frac{1}{\sqrt{2}} ,\]
        the generalized Newton method cycles between $s$ distinct and definite
points when the initial $z$ contains exactly one negative component and
no zeros.
\end{proposition}
\begin{proof}
        Suppose the current approximation \(z = (z_i)_{i=1}^{s}\) consists of only
positive components except for one, say \(z_i < 0\). Then we will show
that the next iterate \(z^+\) has only positive iterates except  for \(
0>z^+_{i^+}\) with \(i^+ \equiv  1\, + \,( i \! \mod \! s )\). This relation
obviously establishes the assertion, since the single negative sign will
cycle infinitely often.
Due to the symmetry of the situation we may assume w.l.o.g. that the last
component of the current iterate  \(z\) is negative. Hence here we have
\(\Sigma(z) = \diag(1,\dots,1,-1)\) and the next iterate \(\zeta = z^+\)
is then the  solution of the system of linear equations,
        \begin{multicols}{2}
        $\quad$\\[-0.8cm]
        \[ \begin{bmatrix} 1 & 0 & \dots & a \\-a & 1 & \dots & 0 \\  & \ddots \\
& &  -a & 1 \end{bmatrix}\begin{bmatrix} \zeta_1 \\ \vdots \\ \\
\zeta_{s} \end{bmatrix} = \begin{bmatrix} 1 \\ \vdots \\ \\ 1
\end{bmatrix} \] %[-0.5cm]
        Thus in terms of \(\zeta_1\) the other components $\zeta_i$ for \(i =
2,\dots,s\)  are given by
        \vspace*{-0.35cm}
        \[ \zeta_i = 1 + a \,  \zeta_{i-1} = \left(\frac{1-a^{i-1}}{1-a}\right) +
a^{i-1}\zeta_1 \]
        \end{multicols}
        \noindent
        Substituting these expressions into the first line of the system we find
%        \[ \zeta_1 + a\, \zeta_{n} = 1 \implies  \zeta_1 + a{\left(\frac{1-a^{n-1}}{1-a}\right) + a^n\, \zeta_1 = 1 \implies \zeta_1 = \frac{1 - a\left(\frac{1-a^{n-1}}{1-a}\right)}{1+a^n}. \]
\[ \zeta_1 + a\, \zeta_{s} = 1 \implies  \zeta_1 + a
\frac{(1-a^{s-1})}{(1-a)}+ a^s\, \zeta_1 = 1 \implies \zeta_1 = \frac{1 -
2\, a +a^s}{(1+a^s)(1-a)}   \; .\]
        Now we want to achieve a shift of the negative entry from the last to the
first position during the iteration from \(z\) to \(z^+\). So
        \(\zeta_1\) should become negative and \(\zeta_2\) has to stay positive.
In other words, we have to impose the two conditions \(\zeta_1 < 0\) and
\(\zeta_2 > 0\). From the first one it  follows that
        \[ 0 > 1 - a\frac{(1-a^{s-1})}{(1-a)} \iff \sum_{i=0}^{s-1} a^i > 2  \]
        and the second one is equivalent to
        \[ 0 < \zeta_2 = 1 + a\, \zeta_1 = 1 + \frac{a}{(1+a^s)}\left[1 -
a\frac{(1-a^{s-1})}{(1-a)})\right] \iff 1+a^s > 2a^2. \]
        The last condition is certainly met by all $a  \leq {1}/{\sqrt{2}} < 1$.
        To ensure the first condition $\zeta_1<0$ we substitute   \( a = \frac 12
(1+\Delta a) \) for some $\Delta a \in (0,\sqrt{2}-1)$.
        Clearly, the first condition is monotonic in $a$ and $\Delta a$ so that,
if it holds for the  particular $\Delta a =2^{1-s}$, it must also hold
for all greater values of that problem parameter. Now we obtain after
some elementary manipulations
        \begin{align*}
                \sum_{i=0}^{s-1} \left[\frac 12 (1+\Delta a)\right]^i > 2 &\iff
\frac{1-\left[\frac 12 (1+\Delta  a)\right]^s}{1-\frac 12 (1+\Delta  a)}
> 2  %\\
                &\iff \Delta  a < 2\sqrt[s]{\Delta  a} - 1.
        \end{align*}
        The only thing that remains to be shown is that the last inequality holds
for  \( \Delta a \equiv 2^{1-s}\). For $s=3$ this is easily verified by
direct calculation. For all $s\geq 4$ we obtain the condition
        $$  2 \sqrt[s]{\Delta a}-1  \; = \; 2^{1/s} - 1 \;  \geq \; 1/(2s) \; .$$
Here, the last inequality holds for $s \geq 2$ since the function $2^{1/s}
- 1-  1/(2s)$ of $s$ is positive for $s=2$ and one can easily check by
differentiation that it grows monotonically beyond.  Now all that remains
to be shown is that $ 2^{s-2} \; < \; 1/s$, which  one can check quite
easily to be indeed satisfied for all $s > 3$.  This completes the proof.
\end{proof}

The proposition demonstrates that, at least without additional structural information on $S$, we cannot deduce the convergence of full step generalized Newton when 
 $\rho(|S|) \in [1/2+1/2^n, \sqrt{2}]$. Please note that this divergence-result does not hold for the method outlined in Corollary \ref{76}. 
 
%  However, since in the cyclic example $\hat c > 0$ and all its components have identical absolute values, the signed Gaussian elimination will perform no pivots and set $\sigma_1 = 1$. The first step of the Gaussian elimination will simply be to add the first row, multipled by $a$, to the second. Then the second component in the updated $\hat c$ will be positve and have, with $1+a$ the largest absolute value in $\hat c$, which again means: no pivoting and a correct sign pick. It is easy to see that this principle carries on, so that after the $i$-th step of the elimination the entry $i+1$ will have the slightly Horneresque value $1+a(1+a(\dots (1+a)))$, whereas the entries from $i+2$ to $s$ will be $1$, ensuring that, at each stage, there are no pivots, but a correct sign pick performed until the elimination is complete. Which clearly shows that the methods are not equivalent.  

 Also, because our  fixed point iteration and the modulus method  normally yield only linear convergence, it becomes immediately clear that they do not reduce to semi-smooth Newton. Under the assumption of smooth dominance the local convergence result of Qi et al. applies and we must have finite convergence on PL problems whenever convergence occurs at all. Of course, evaluating $\hat H(z)$ is a lot cheaper than solving a system in the Jacobian 
$J_\sigma  =  J  + Y \Sigma(I-L\Sigma)^{-1}Z$ with $\sigma  = \sigma(x)$ and thus $\Sigma = \Sigma(x)$, changing from iterate to iterate. 
While the iteration function  $G$ is Lipschitzian, the not always unique generalized Newton steps $-J_{\sigma(x)}^{-1}F(x)$ may jump discontinuously as a function of $x$. 
Nevertheless, it might be worthwhile to switch to Newton once the signature vector $\sigma$ has been stable for a few iterations.

 It is not too hard to see that (at least when full steps are taken) the generalized Newton iteration on $H(z)$ is equivalent to that applied to the 
 partitioned equation \eqref{absnormal} for fixed $y$. The key numerical effort is solving a linear system in $I - S\, \Sigma$, which is also the key effort in 
 applying the inverse Jacobians $J_\sigma^{-1}$ to any vector. In either case we first need to form the Schur complement $S$, which, at least formally, 
 involves the inverse of the smooth part $J$. If the number $s$ of switching variables is much smaller than $n$, the number of independents, we can of course 
 compute $J^{-1}Y$ or $ZJ^{-1}$ by solving $s$ linear systems in $J$, possibly based on its $LU$ factorization.      
 
When $H(z)$ is injective, the fibres \eqref{fibre} have no bifurcations at all, so tracing them in a piecewise Newton fashion seems a very promising approach.  Naturally, the number of steps is not a priori bounded in any way. 
To see that this is not equivalent to applying piecewise Newton to the original system $F(x)=0$ we note that in the latter case, until the final step, there will always be a nontrivial residual on the lower  equation of \eqref{absnormal}, whereas the upper block will be exactly satisfied.  Conversely, applying piecewise Newton to $H(z)=\hat c$ means that  there will be a residual in the upper block but the lower equation will remain exactly satisfied. Of course, one could also try a mixture  just starting from $(x,z)=(0,0)$ so that all subsequent residuals would be multiples of 
$(c,b)$. The advantages and disadvantaged of these approaches deserve  to be explored in detail.

\subsection*{Reduction to an LCP}
 %\subsection{Equivalence to linear complementarity problem} 
 Decomposing $z=u-w$ with $u \perp w$ in that $u \geq 0 \leq w$ and $u^\top w =0$, we obtain $|z| = u+w$. Substituting this into our basic equation for fixed $y$, and subsuming $y$ into $b$, we obtain 
 \begin{eqnarray}
 \begin{bmatrix} u-w  \\ 0  \end{bmatrix} \;  = \; \begin{bmatrix} c \\  b \end{bmatrix} +   \begin{bmatrix} Z & L \\ J & Y   \end{bmatrix} \; 
\begin{bmatrix} x \\  u+w  \end{bmatrix}   \quad \mbox{with} \quad 0 \leq u \perp w \geq 0\label{reducednorm2}  \; .
\end{eqnarray} 
 Assuming again that the smooth part $J$ is nonsingular we can eliminate $x$ using the second equation and obtain with $S$ the Schur complement as above with the abbreviation $\hat c \equiv  c - Z J^{-1}b $
  $$  u-w  \;  = \; \hat c + S (u+w)    \quad \mbox{with} \quad 0 \leq u \perp w \geq 0 \; .$$
  Assuming furthermore that $I-S$ is nonsingular, which is certainly implied by smooth dominance, we may solve for $u$ and obtain 
  
  \be
  0 \leq  u \;  \equiv \;  q + M \, w \; \,  \perp \; \,  w \geq 0 \label{lcp} \ee
  where
  \be 
   q \; \equiv (I-S)^{-1} \hat c \quad \mbox{and} \quad  M \; \equiv \; (I-S)^{-1} (I+S) \label{defM} \; .
   \ee
 This is a linear complementarity problem in standard form. Of course, in this transformation some sparsity and structure of the original piecewise equation may be lost. Nevertheless, we 
  should keep in mind that, when the  smooth Jacobian $J$ is invertible and the Schur complement $S$ does not  have the eigenvalue $1$, then we are essentially solving a complementarity problem in $s$ variables. 
   If $S-I$ but not $S+I$ is singular we can exchange the roles of $v$ and $w$ to get essentially the same reduction with $M$ being the inverse of its definition above. 
  %Moreover if $S$ is smoothly dominant, it follows from Neumann series that $S=(I-M)^{-1}(I+M)$. 
  %Then the simple fixed point iteration mentioned above \ref{18a} becomes fully identical to the one from Theorem $9$ in \cite{van1981piecewise}.
  %\newpage
  Rather than eliminating the vector $x$ we could also split it into  complementary positive and negative parts. However, especially  since $J$ can always be made nonsingular using (\ref{trivial}) 
  essentially doubling $x$ would seem to introduce artificial combinatorial complexity.
 Since every solution of our complementary equation $H(z)=\hat c$ corresponds to a solution of the LCP, the latter can be uniquely solved for any vector $q$
 if we have smooth dominance. It is well known \cite{0757.90078} that this is true if and only if $M$ is a P-matrix.  On the other hand, Rump has shown that 
 $\rho_0^s(S) < 1 $ is equivalent to $M$ being a P-matrix, which agrees with our bijectivity result for simply switched coherently oriented systems.  
 % The LIKQ condition is trivially satisfied by  the CPL $H(z)=\hat c$ since for that PL function $Z=I$.        
 % Now let us examine assumptions on the PL problem in the form (\ref{absnormal}), under which  the complementarity matrix 
  %$ M \; \equiv \; (I-S)^{-1} (I+S)  $ satisfies the various solvability conditions for LCPs. 
 % \begin{corollary}
 % If $\det(J) \neq 0$ and we have coherent orientation of H then the $M$ defined in (\ref{Mdef}) is a P matrix.
  % \end{corollary}
  {}
 
 {\iffalse 
 We also may show that more directly. According to Lemma xx in \cite{cottle} $M$ is a P-matrix exactly if for all $z \in \R^s$ 
 $$  z \geq 0 \quad \mbox{and} \quad  w \equiv M z \leq 0  \quad \Longrightarrow   \quad z = 0 = w $$
With $M = (I-S)^{-1}(I+S)$ we derive from the left-hand side that 
$$(I-S)w = (I+S)z \quad \mbox{with} \quad  w\leq 0 \quad \mbox{and} \quad  z\geq 0$$ 
That implies 
$$ 0 \leq z - w = |z|+|w| =  S (z+w)  \leq  |S|  (|z|+|w|) \leq |S|^k   (|z|+|w|) $$
for any $k >1$. Due to the spectral radius of $|S|$ being less than $1$ this implies $z=0=w$ and thus the P-matrix property of $M$.       
  \end{proof}
 \fi}
 %\newpage  
 \vspace{-0.3cm}
   \section{Summary and Outlook}   
    In this paper we have examined the properties of piecewise linear functions that are given in abs-normal form. Such a representation is always possible, but by no means unique. A key quantity is the switching depth $\nu$, which we conjecture to be reducible to the bound  $\bar \nu(n) = 2\, n-1$. Of particular importance is the case of $\nu=1$, where we call $F$ simply switched. If such a representation exists, it is shown here that openness and bijectivity coincide provided LIKQ or the slightly  weaker nondegeneracy condition of stable coherent orientation is satisfied.   
 
The Schur complement matrix $S = L - Z J^{-1} Y$, whose existence  depends on the nonsingularity of the smooth part $J$, plays a central role throughout. 
In particular it yields the complementary system $H(z) = [I- S \Sigma]z = \hat c$. This piecewise linear function $H(z)$ is simply switched and satisfies the LIKQ condition. 
Hence it is, according to Proposition 4.1, injective if and only if it is coherently oriented, which, in turn, is equivalent to the      
the signed real spectral radius of $S$ being less than 1. 
In principle this can be tested, though the  evaluation of the continuous function $\rho^s_0(S)$ is generally NP hard as shown in \cite{rump1997theorems}. Since injectivity of $H(z)$ implies injectivity of the underlying $F(x)$ the condition  $\rho^s_0(S) < 1 $ is also sufficient for injectivity of $F(x)$. However, we have as yet no practical criterion for $F(x)$ to be merely open other than the theoretical possibility of exhaustively checking all Jacobians of $F$.  Such combinatorial procedures have otherwise  been avoidable throughout, thanks to the representation of $F$ in abs-normal form.   

The key properties form the following chain of implications:%\vspace{-0.3cm}
\begin{center}
%$ \qquad\qquad \Longrightarrow  \; $ \\
 Absolute Contractivity $\;   {\bf \Longrightarrow}  \;$ Smooth Dominance $  \; {\bf \Longrightarrow} \;  $ Bijectvity of $H$
 $$\qquad  \rho(|S|) < 1 \qquad \qquad \qquad   \| D S D^{-1}\|_p <1 \qquad \qquad  \qquad   \rho_0^s(S) < 1 . $$
\end{center}
%\newpage 
So far our Linear Independence Kink Qualification (LIKQ) has only been defined in the simply switched case and it is then equivalent to the familiar linear independence constraint qualification (LICQ). However, we believe there is a generalization to PL problems, where the kinks do not even locally consist of a set of 
intersecting hyperplanes, as is often envisioned. Instead, there is a hierarchy of kinks with the later ones being broken into affine pieces by the earlier ones. The algorithmic handling of this structure is still not entirely clear, even in the context of minimizing a scalar valued PL function.    

In order to constructively solve PL systems of equations 
one may apply full-step or piecewise Newton to either the original  problem $F(x)=0$ or the complementary version $H(z)= \hat c$. They are guaranteed to converge if
 $S$ does not deviate too much from $L$, which ensures at least coherent orientation.  
 More specifically, we obtain finite convergence of generalized  Newton on $H(z)=\hat c$ when 
 $\|S\|_p < 1/3$ or $\rho(|S|) < 1/2$. The second bound is quite sharp in that divergence can occur as soon as $\rho(|S|) \geq 1/2+1/2^n$, as demonstrated in 
 Proposition 7.5. 
    
 Apart from these four variants one may also apply damped versions or the fixed point iteration \( z^+ = \hat c + S\lvert z\rvert \), provided one has smooth dominance, i.e., $\|S\|_p < 1$  for some $p\geq 1$, which is stronger than coherent orientation of $H$ and thus injectivity of $F, H$.  Piecewise smooth problems can be solved by successive piecewise linearization, yielding at least locally quadratic convergence. In this context coherent orientation of the piecewise linear model near the current outer iterate should be sufficient. %\\[0.9cm] 
 
Abbreviating $\hat \rho = \|J^{-1}Y\|_p\|Z\|_p$ we may compile the table of solvers listed in Table \ref{methodstable}. The effort column shows, which linear systems need to be solved, usually once per iteration. In the signed Gaussian elimination the equivalent of just one single solve is needed.   
{\iffalse    Let \(n=1000\), then again consider the Matrix \(S = A_n\) from (\ref{eqn:RumpI}) and vectors \(\mathbf{c} = (\sin(i))_{i=1\dots n},\, z_0 = \mathbf{0} \in \R^n\). We have already seen that the complementary system \(H(z) = (I_n-S\Sigma_z)z\) is coherently oriented, but not smoothly dominant. The fixpoint iteration \( z^+ = \hat c + S\lvert z\rvert \) from starting point \(z_0\) diverges.
\fi}
\begin{table}[H]
\hspace*{-0.20cm}
\begin{tabular}{| l | l |c| c| l } 
\hline
Method & Convergence condition & Rate & Effort \\
\hline
Generalized Newton on OPL & $ 2 \, \hat \rho < (1\! - \! \|L\|_p \! - \! \hat \rho/2)^2 $& finite & $I \! - \! S\Sigma, J$ \\ \hline 
Generalized Newton on CPL &  $\|S\|_p < 1/3$ & finite & $I\! -\! S\Sigma$ \\ \hline 
Signed Gauss on CPL &  $ \rho(|S|) < 1/2$ & finite & $I\! -\! S\Sigma$  once \\ \hline 
Block Seidel on CPL & $\|S-L\|_p + \|L\|_p < 1$ & linear & $I \! - \! L\Sigma, J$  \\ \hline
Modulus Iteration on CPL & $\|S\|_p < 1$ & linear & J \\ \hline 
Piecewise Newton on OPL & coherent orient. of F & finite & $I\! - \!  S \Sigma, J$ \\ \hline 
Piecewise Newton on CPL & $\rho_0^s(S) < 1$ & finite & $I \! - \! S \Sigma$ \\ \hline
\end{tabular}
\caption{Solvers for PL systems of equations in original abs-normal or complementary form.}  
\label{methodstable}
\end{table} 
%$\quad $ \\[0.5cm]

%\newpage
Another theoretical possibility is piecewise Newton on the combined system in terms of $x$ and $z$. A more promising approach would appear to be the combination of  the fixed point iterations with Newton variants, which should yield finite convergence if one can get into the vicinity of a root.    
%\vspace*{1cm}
Without coherent orientation the fibres $\{F(x) = \lambda F(x_0) : \lambda > 0 \}$ and also $\{H(z) -\hat c = \lambda (H(z_0)-\hat c) : \lambda >0  \}$ may contain turning points, which could be followed by some version of Branin's method \cite{branin1972widely} originally defined by
$$     \dot x \; = \; \pm \mathbf{adj}(F^\prime(x)) F(x) \quad \mbox{with} \quad \det(F^\prime(x)) I \; = \;  F^\prime(x)\,  \mathbf{adj}(F^\prime(x))  \; .$$
In the general smooth case such trajectories may converge to roots, cycle or run off to infinity. Possibly the inherent finiteness of PL functions makes it possible to avoid some of these calamities. Other globalized searches remain to be investigated. 
Since any Lipschitzian vector function may be approximated on compact domains by PL functions, there can be no magic solver for the general case.  Numerical experiments with the various methods considered here are currently under way.

\section*{Acknowledgements}
The proof of Proposition 4.2 and Lemma 4.3 was thankfully provided by our colleague Dorothee Sch\"uth of Humboldt University. 
The authors are  also indebted to Daniel Kressner, who pointed out the connection between the coherence condition $\det(I- \Sigma \, S)>0 $ and the sign real spectral radius
$\rho^s_0(S)$ of the Schur complement $S$ being less than $1$.  They are also grateful to Torsten Bosse, who contributed many insights into the piecewise linearization approach and  greatly helped with the composition of this article. Finally, the paper benefited greatly from the corrections and suggestions of the two anonymous referees. 

%\newpage 
 \bibliographystyle{alpha}
\bibliography{autodiff,../junkfinal/ad,../final/secant}
\end{document}